\documentclass[12pt,centertags,oneside]{amsart}
\usepackage{amsmath,amstext,amsthm,amscd,typearea}
\usepackage{amssymb}
\usepackage{a4wide}
\usepackage[mathscr]{eucal}
\usepackage{mathrsfs}
\usepackage{typearea}
\usepackage{charter}
\usepackage{pdfsync}
\usepackage{url}

\usepackage[a4paper,width=16.2cm,top=3cm,bottom=3cm]{geometry}

\numberwithin{equation}{section}

\newtheorem{theorem}{Theorem}[section]
\newtheorem{definition}[theorem]{Definition}
\newtheorem{proposition}[theorem]{Proposition}
\newtheorem{corollary}[theorem]{Corollary}
\newtheorem{lemma}[theorem]{Lemma}
\newtheorem{remark}[theorem]{Remark}

\newcommand{\cali}[1]{\mathscr{#1}}

\newcommand{\supp}{{\rm supp}}

\newcommand{\Int}{{\rm Int}}

\newcommand{\dist}{\mathop{\mathrm{dist}}\nolimits}
\renewcommand{\Re}{\mathop{\mathrm{Re}}\nolimits}
\renewcommand{\Im}{\mathop{\mathrm{Im}}\nolimits}

\newcommand{\B}{\mathbb{B}}
\newcommand{\C}{\mathbb{C}}
\newcommand{\D}{\mathbb{D}}
\newcommand{\N}{\mathbb{N}}

\newcommand{\R}{\mathbb{R}}

\newcommand{\boldsym}[1]{\boldsymbol{#1}}

\title[]{Complex Monge-Amp\`ere equation for measures supported on real submanifolds}

\author{Duc-Viet Vu}
\address{UPMC Univ Paris 06, UMR 7586, Institut de
Math{\'e}matiques de Jussieu-Paris Rive Gauche, 4 place Jussieu, F-75005 Paris, France.}
\email{duc-viet.vu@imj-prg.fr}
\thanks{This research is supported by grants from R\'egion Ile-de-France. }

\date{\today}
\begin{document}

\begin{abstract} Let $(X,\omega)$ be a compact $n$-dimensional K\"ahler manifold on which the integral of $\omega^n$  is $1$. Let $K$ be an  immersed real  $\mathcal{C}^3$ submanifold of $X$ such that the tangent space at any point of $K$ is not contained in any complex hyperplane of the (real) tangent space at that point of $X.$ Let $\mu$ be a probability measure compactly supported on $K$ with $L^p$ density for some $p>1.$ We prove that the complex Monge-Amp\`ere equation $(dd^c \varphi + \omega)^n=\mu$ has a H\"older continuous solution. %$\varphi$ provided that $\mu$ is a normalized Riemannian volume form of a compact real $\mathcal{C}^3$ submanifold of $X$ whose tangent spaces are not contained in any complex hyperplane of the ones of $X.$
\end{abstract}

\maketitle

\medskip

\noindent
%{\bf Classification AMS 2010}: 

\medskip

\noindent
{\bf Keywords:} Monge-Amp\`ere equation,  generic CR submanifold.   

\tableofcontents

%%%%%%%%%%%%%%%%%%%%%%%%%%%%%
%%%%%%%%%%%%%%%%%%%%%%%%%%%%%%%%%%
\section{Introduction} \label{introduction}

Let $X$ be a compact K\"ahler manifold of dimension $n$ and let $\omega$ be a fixed K\"ahler form on $X$ so normalized that $\int_X \omega^n=1$. The aim of this paper is to give a useful explicit class of measures for which the complex Monge-Amp\`ere equation has a H\"older continuous solution. Recall that a real $\mathcal{C}^1$ manifold $K$ is said to be \emph{immersed} in $X$ if there is an injective $\mathcal{C}^1$ immersion from $K$ to $X.$ In this case we say that $K$ is an immersed $\mathcal{C}^1$ submanifold of $X.$ An immersed \emph{real} $\mathcal{C}^1$ submanifold $K$ of $X$ is said to be \emph{generic CR} (or \emph{generic} for simplicity) in the sense of the Cauchy-Riemann geometry if the tangent space at any point of $K$ is not contained in a complex hyperplane of the tangent space at that point of $X.$ Such a submanifold has the real dimension at least $n.$ %Let $K$ be a generic compact submanifold. Let $vol_K$ be a Riemannian volume form of $K.$  
A function $\varphi: X \rightarrow [-\infty, \infty)$ is \emph{quasi-p.s.h.} if it is locally the sum of a p.s.h. function and a smooth one. A quasi-p.s.h. function is said to be \emph{$\omega$-p.s.h.} if  we have $dd^c \varphi + \omega \ge 0$ in the sense of currents.  The following is our main result. 

\begin{theorem} \label{the_MAgeneric} Let $K$ be a generic immersed $\mathcal{C}^{3}$ submanifold of $X$ of  real codimension $d>0.$ Let $\mu$ be a probability measure compactly supported on $K$ with $L^p$ density for some $p>1.$ Then the Monge-Amp\`ere equation $(dd^c \varphi + \omega)^n=\mu$ has an $\omega$-p.s.h. solution $\varphi$ which is H\"older continuous with H\"older exponent $\alpha,$ for any positive number $\alpha< \frac{2(p-1)}{3d(n+1)p} \cdot$  
\end{theorem}

Note that our proof still holds if $K$ is $\mathcal{C}^{2,\beta}$ for some $\beta \in (0,1).$ In this case one just needs to replace the $\mathcal{C}^{2,1/2}$ regularity in Section \ref{sec_analyticdisc}  by $\mathcal{C}^{2,\beta'}$ one for $\beta' \in (0,\beta)$. For simplicity, we only consider the $\mathcal{C}^3$ regularity as in Theorem \ref{the_MAgeneric}.  %Secondly, by \cite[Cor. 4.5]{DinhVietanhMongeampere}, the conclusion of Theorem \ref{the_MAgeneric} still holds if we replace $vol_K$ by $f vol_K$ where $f$ is a positive function  in $L^p(K, vol_K)$ with $p>1$ and $\|f\|_{L^p(vol_K)}=1.$ This gives us an explicit rich class of Monge-Amp\`ere measures having H\"older continuous potentials. 
Secondly, if the Monge-Amp\`ere equation has a H\"older continuous solution, then that solution is unique up to an additive constant. This is a direct consequence of results in  \cite{Kolodziej_Acta,Dinew_09}.  %From the point of view of the intersection of positive closed currents, Theorem \ref{the_MAgeneric} gives us a rough idea about the supports of self-intersections of positive closed $(1,1)$-currents having H\"older continuous potentials.   

For a probability measure $\mu$ on $X,$ the associated complex Monge-Amp\`ere equation 
\begin{align} \label{eq_MA}
(dd^c \varphi + \omega)^n=\mu
\end{align}
has been extensively studied since the fundamental paper \cite{Yau1978} of Yau in which he  proved that (\ref{eq_MA}) has a unique smooth solution if $\mu$ is a (smooth) Riemannian volume form $vol_X$ of $X.$ Later Ko{\l}odziej showed that  the Monge-Amp\`ere equation admits a unique continuous solution for a larger class of measures $\mu$ which contains $\mu=f vol_X$ with $f \in L^p(X)$ for $p>1,$ see \cite{kolodziej05,Kolodziej_Acta}. For the last measures, he also obtained H\"older regularity of the solution in \cite{Kolodziej08holder}. The H\"older exponent of that solution is then made precise by Demailly, Dinew, Guedj, Hiep, Kolodziej and Zeriahi  in \cite{DemaillyHiep_etal} using the regularization method in \cite{Demailly_appro_chernconnec} and the stability theorem in \cite{Dinew_Zhang_stability}. Moreover, in \cite{Hiep_holder} Hiep obtains the H\"older regularity for $\mu=f vol_Y,$ where $vol_Y$ is the volume form of a compact real hypersurface $Y$ of $X$ and $f \in L^p(Y)$ for $p>1.$

Recently, Dinh and Nguy\^en  in \cite{DinhVietanhMongeampere}  show that the class of probability measures $\mu,$ for which (\ref{eq_MA}) admits a H\"older continuous solution, is exactly the class of probability measures whose super-potentials are H\"older continuous, see Definition \ref{def_superholder} below. They then recover the aforementioned results in \cite{Kolodziej08holder,Hiep_holder,DemaillyHiep_etal}. By \cite{DinhVietanhMongeampere}, we know that if a probability measure $\mu$ having a H\"older continuous super-potential of order $\beta \in (0,1],$ then the solution of (\ref{eq_MA}) is H\"older continuous of order $\beta'$ for any $0 < \beta' <2 \beta/(n+1).$    For more information on the complex Monge-Amp\`ere equation, the readers may consult the survey \cite{Phong}.

Theorem \ref{the_MAgeneric} above  combined with \cite[Pro. 4.4]{DinhVietanhMongeampere} yields the following nice exponential estimate, see also  \cite{Skoda_integrability,DVS_exponential,Lucas_ex}.

\begin{corollary} Let $K$ be a generic immersed $\mathcal{C}^{3}$ submanifold of $X$. Let $\tilde{K}$ be a compact subset of $K.$ Then the restriction of the Lebesgue measure on $K$ to $\tilde{K}$ is moderate, that is, there exist two positive constants $\alpha$ and $c$ such that  for any $\omega$-p.s.h. function $\varphi$ on $X$ with $\sup_X \varphi =0$ we have 
$$\int_{\tilde{K}} e^{- \alpha \varphi} d\, vol_K \le c.$$ 
\end{corollary}

Before presenting the idea of the proof of Theorem  \ref{the_MAgeneric}, we need to recall  some definitions.  Let $\mu$ be a  probability measure on $X.$ Let $\cali{C}$ be the set of $\omega$-p.s.h. functions $\varphi$ on $X$ such that $\int_X  \varphi \, \omega^n= 0.$ We define the distance  $\dist_{L^1}$ on $\cali{C}$ by putting 
$$\dist_{L^1}(\varphi_1,\varphi_2):= \int_X |\varphi_1 - \varphi_2| \, \omega^n,$$ 
for every $\varphi_1, \varphi_2 \in \cali{C}.$ 

\begin{definition} \label{def_superholder} The super-potential of $\mu$ (of mean $0$) is the function $\cali{U}:  \cali{C} \rightarrow \R$ given by $\cali{U}(\varphi):= \int_X \varphi d\mu.$ We say that $\cali{U}$ is H\"older continuous with H\"older exponent $\alpha \in (0,1]$ if it is so with respect to the distance $\dist_{L^1}.$     
\end{definition}

By \cite[The. 1.3, Cor. 4.5]{DinhVietanhMongeampere}, Theorem  \ref{the_MAgeneric} is a direct consequence of the following result.

\begin{theorem} \label{the_MAgeneric2} Let $K$ be a generic immersed $\mathcal{C}^{3}$ submanifold of $X$ of real codimension $d>0$. Let $\tilde{K}$ be a compact subset of $K$ and $\boldsym{1}_{\tilde{K}}$ the characteristic function of $\tilde{K}.$  Let $vol_K$ be an arbitrary $\mathcal{C}^3$ Riemannian volume form of $K.$ Then the super-potential of $\boldsym{1}_{\tilde{K}} vol_K$ is H\"older continuous with H\"older exponent $\alpha$ for any positive number $\alpha < 1/(3d)$.
\end{theorem}

Let $\D$ be the unit disc in $\C$ and let $\partial \D$ be the boundary of $\D.$ A \emph{$\mathcal{C}^1$ analytic disc} in $X$ is a $\mathcal{C}^1$ map from $\overline{\D}$ to $X$ which is holomorphic on $\D.$ For a nonempty arc $I \subset \partial \D,$ an analytic disc $f$ is said to be \emph{$I$-attached to a subset $K$} of $X$ if $f(I)$ belongs to $K$. When we do not want to mention $I,$ we simply say an analytic disc partly attached to $K.$  Throughout this paper,  for every parameter $\tau,$ we will systematically use the notation $\lesssim_{\tau}$ or $\lesssim$ which means $\le$ up to a constant depending only on  $(\tau,X, K, \omega)$ or on $(X,K,\omega)$ respectively. A similar convention is applied to $\gtrsim_{\tau}$ and $\gtrsim.$ %And when we use   for a parameter $\tau,$ we also $ \lesssim_{\tau}$ or $\gtrsim_{\tau}$ to indicate the dependence of the last constant on $\tau.$  %In order to prove Theorem \ref{the_MAgeneric},  by results mentioned early in \cite{DinhVietanhMongeampere} it is enough to prove that the super-potential of $vol_K$ is H\"older continuous.  

The idea of the proof of Theorem  \ref{the_MAgeneric2} is as follows. Observe that the codimension $d$ of $K$ is at most equal to $n.$ We consider below the case where $d=n.$ The other cases can be deduced from it.  Let $\varphi_1,\varphi_2 \in \cali{C}$ and $\varphi:= \varphi_1 -\varphi_2.$ To show the H\"older regularity of the super-potential of $vol_K,$ by definition we need to bound the $L^1$-norm of $\varphi$ with respect to $vol_K$ by a power of the $L^1$-norm of $\varphi$ on $X.$ Since one can approximate any $\omega$-p.s.h. function on $X$ by a decreasing smooth ones (see  \cite{Blocki_kolo_regular}), it is enough to prove the desired property for smooth  $\varphi_1, \varphi_2$ with $\varphi_1 \ge \varphi_2,$ see Proposition \ref{pro_equitheoremMAgeneric} and Lemma \ref{le_reduction_nonnegative}. In this case, $\varphi$ is \emph{smooth and nonnegative.} This reduction is crucial in our proof.  Observe that by compactness of $\tilde{K},$ it suffices to estimate 
$$\int_{\tilde{K}'} \varphi  \,dvol_K,$$
for small open subsets $\tilde{K}'$ of $\tilde{K}$. For each point $a \in K,$ we will construct a  $\mathcal{C}^{2,1/2}$-differentiable family $\tilde{F}_{ \{\boldsym{\tau} \in Z\}}$ of analytic discs partly attached to $K$ parameterized by $\boldsym{\tau}$ in a  compact manifold $Z$ of real dimension $(2n-2)$ which roughly satisfies the following two properties:

$(i)$ the restriction of $\tilde{F}$ to $\partial \D \times Z$ is a submersion onto an open neighborhood $K'$ of $a$ in  $K,$ where we consider $\tilde{F}$ as a map from $\overline{\D} \times Z$ to $X.$

$(ii)$ the restriction of $\tilde{F}$ to $\D \times Z$ is a diffeomorphism onto  an open subset of $X.$\\
Put $\tilde{K}':= \tilde{K} \cap K'$ for $a \in \tilde{K}.$ These $\tilde{K}'$ covers $\tilde{K}.$ By the change of variables theorem and Property $(i)$, we have 
\begin{align} \label{ine_KF*}
\int_{\tilde{K}'} \varphi\, d vol_K \le    \int_{K'} \varphi\, d vol_K \lesssim \int_{\partial \D \times Z} \varphi \circ \tilde{F}.
\end{align}
Since $\tilde{F}$ is holomorphic on $\D$ and $\mathcal{C}^2$ on $\overline{\D},$ observe that $\varphi \circ \tilde{F}$ is the difference of two $\mathcal{C}^2$ subharmonic functions on $\overline{\D}.$ %which belong to a fixed compact subset $\cali{C}_{\tilde{F}}$ of the set of continuous subharmonic functions on $\overline{\D}$ in $L^1$-topology.  

Our second step is to bound $\int_{\partial \D \times Z}  \varphi \circ \tilde{F}$ by a quantity involving $\int_{\D \times Z} \varphi \circ \tilde{F}.$ For this purpose, we will establish a crucial inequality in dimension one which shows that $L^1$-norm on $\partial \D$ of a nonnegative $\mathcal{C}^2$ function on $\overline{\D}$ is bounded by a function of  its $L^1$-norm on $\D$ and some H\"older norm of its Laplacian on $\D$. The ingredients for the proof of the last inequality are  Riesz's representation formula and a general interpolation inequality for currents on manifolds with boundary. Note that a version of that interpolation inequality for manifolds without boundary was firstly used by Dinh and Sibony in \cite{DinhSibony_Pk_superpotential}.   

The problem will be solved if one is able to bound  $\int_{\D \times Z} \varphi \circ \tilde{F}$ by a constant times  $\|\varphi\|_{L^1(X)}.$ Taking into account Property $(ii)$, one is tempted to use the change of variables  by $\tilde{F}.$ However, the Jacobian of $\tilde{F}$ is small  near the boundary $\partial \D \times Z.$   This is due to a general fact that any family of analytic discs satisfying Property $(i)$ should degenerate at $\partial \D$ because of its attachment to $K.$ So we need a precise control of the Jacobian of $\tilde{F}$ from below and prove some estimates on the integrals of p.s.h. functions and their $dd^c$ on a tubular neighborhood of $\tilde{K}$. These estimates are of independent interest.  Consequently, we will get 
\begin{align} \label{ine_XF*}
\int_{\D \times Z} \varphi \circ \tilde{F} \lesssim_{\alpha_2} \bigg(\int_{X} \varphi \, d vol_X\bigg)^{\alpha_2},
\end{align}  
 for any $\alpha_2 \in (0, 1/n).$ Combining these above inequalities gives the H\"older regularity of the super-potential of $vol_K.$

The paper is organized as follows. Section \ref{sec_interpolation} is devoted to proving the above mentioned interpolation inequality for currents. In Section \ref{sec_analyticdisc}, we construct the desired family of analytic discs $\tilde{F}.$ In Section \ref{sec_intergrability}, we present (\ref{ine_KF*}) and (\ref{ine_XF*}). Finally, we prove Theorem  \ref{the_MAgeneric2} in Section \ref{sec_superpotential}. At the beginning of Section \ref{sec_analyticdisc}, we will fix  some notations which will be used for the rest of the paper.

\vskip 0.5cm
\noindent
{\bf Acknowledgement.} The author  would like to thank Tien-Cuong Dinh for introducing him this research topic and for his illuminating discussions. He also wants to express his gratitude to Lucas Kaufmann for fruitful discussions.

\section{Interpolation theory} \label{sec_interpolation}

Let $M$ be  a compact smooth manifold of dimension $m$. Fix a partition of unity subordinated to a finite covering of local charts of $M.$ For $k \in \N$ and $\alpha \in (0,1],$ let $\mathcal{C}^{k, \alpha}(M)$ be the space of $\mathcal{C}^k$ functions on $M$ whose partial derivatives of order $k$ are H\"older continuous of order $\alpha.$ We endow the last space with the usual norm.  For $t\in [0,\infty),$ denote by  $\mathcal{C}^{t}(M)$ the space $\mathcal{C}^{[t], t- [t]}(M)$ where $[t]$ is the integer part of $t.$ Let $\Lambda^l T^* M$ be the $l^{th}$-exterior power of the cotangent vector bundle $T^*M$ for $1 \le  l \le m.$ Let $\mathcal{C}^t(M,\Lambda^l T^* M)$ be the set of $l$-differential forms with $\mathcal{C}^l$ coefficients. Using the above fixed partition of unity, we can equip $\mathcal{C}^t(M,\Lambda^l T^* M)$ with the norm $\|\cdot \|_{\mathcal{C}^t}$ which is the maximum of the $\mathcal{C}^t$ norms of its coefficients.

 Let $T$ be an $l$-current of order $0,$ \emph{i.e.}, there is a constant $C$ such that  $|\langle T, \Phi \rangle| \le C \|\Phi\|_{\mathcal{C}^0}$  for every smooth $(m-l)$-form $\Phi.$ For $t\in [0,\infty),$ define 
\begin{align} \label{def_Ctrualpha}
\| T\|_{\mathcal{C}^{-t}}:= \sup_{\Phi \text{ smooth }, \|\Phi\|_{\mathcal{C}^{t}}=1} |\langle T, \Phi \rangle |.
\end{align} 
We will write $\|T\|$ instead of $\| T\|_{\mathcal{C}^{-0}}$ which is the usual mass norm of $T.$  Dinh and Sibony in \cite{DinhSibony_Pk_superpotential} proved that for any $t_1, t_2 \in (0,\infty)$ with $t_1< t_2,$ we have
\begin{align} \label{def_Ctrualpha2}
\| T\|_{\mathcal{C}^{-t_2}} \le \| T\|_{\mathcal{C}^{-t_1}} \le   c\|T\|^{1- t_1/t_2} \| T\|_{\mathcal{C}^{-t_2}}^{t_1/ t_2},
\end{align} 
for some constant $c$ independent of $T.$   This inequality is very useful when dealing with continuous functionals on differential forms because one can reduce the problem to the smooth case. In this section, we will establish a generalization of (\ref{def_Ctrualpha2}) for compact smooth manifolds with boundary. 

Let $M$ be a compact smooth manifold of dimension $m$ with boundary. Cover $M$ by a finite number of local charts $U_j$. Take a partition of unity $\phi_j$ subordinated to this covering. By the aid of these $\phi_j,$ as above we can define the Banach spaces $\mathcal{C}^t (M)$ with the usual norms for $t\in [0,\infty).$   Denote by $\Int M$ the interior of $M.$ Let $\mathcal{C}^{t}_c(\Int M)$ be the 
subspace of  $\mathcal{C}^t(M)$ of  $f \in \mathcal{C}^{t}(M)$ with compact support in $\Int M.$   Let $\tilde{\mathcal{C}}^{t}(M)$ be the subspace of $\mathcal{C}^{t}(M)$ 
consisting of $f$ with $f|_{\partial M} \equiv 0.$ We can also define $\tilde{\mathcal{C}}^t(M,\Lambda^l T^* M)$ and $\mathcal{C}_c^t(M,\Lambda^l T^* M)$ in the same way as above.

Let $T$ be an $l$-current of order $0$ on $\Int M.$ %Define 
%\begin{align} \label{def_Ctrualpha002}
%\| T\|_{\mathcal{C}^{-t}(\Int M)}:= \sup_{\Phi \text{ smooth }, \|\Phi\|_{\mathcal{C}_c^{t}(\Int M)}=1} |\langle T, \Phi \rangle |,
%\end{align} 
%where the supremum is taken over all $\Phi \in \mathcal{C}^{\infty}_c (\Int M,\Lambda^{m-l} T^* M)$ with $\|\Phi\|_{\mathcal{C}_c^{t}(\Int M)}=1.$Let  $\rho$ be the continuous function from $M$ to $\R^+$ by defining $\rho(x)$ to be the distance from $x \in M$ to the boundary of $M.$ Let $\alpha_0$ be a fixed constant in $[0,1).$ Then $\rho^{\alpha_0}\, T$ is also a current of order $0$ on $\Int  M.$ 
Assume that its mass is finite, that is,  
\begin{align} \label{ine_assumptionT}
\|T\|:=  \sup_{\Phi \text{ smooth }, \|\Phi\|_{\mathcal{C}_c^{0}(\Int M)}=1} |\langle T, \Phi \rangle |< \infty.
\end{align}
In our application, $M$ will be $\overline{\D}$ and $T$ will be  the restriction of a continuous form on $\C$ to $\D.$ By Riesz's representation theorem, $T$ is a differential form whose coefficients are Radon measures on $M$ with finite total variations. Hence, for any continuous differential form $\Phi$ on $\Int M$ with $\|\Phi\|_{\mathcal{C}^0}< \infty,$ the value of  $T$ at $\Phi$ is well-defined. Then the current $T$ can be extended to be a continuous linear functional on $\tilde{\mathcal{C}}^t(M,\Lambda^l T^* M).$ Let    $\| T\|_{\tilde{\mathcal{C}}^{-t}(M)}$ be the norm of $T$ as a continuous linear functional on $\tilde{\mathcal{C}}^t(M,\Lambda^l T^* M)$. As mentioned at the beginning of the section, we will prove the following analogue of  (\ref{def_Ctrualpha2}).

\begin{proposition} \label{pro_sosanhdistancevoibien} Let $T$ be a $l$-current of order $0$ on $\Int M.$  Assume that $T$ has finite mass. Let $t_0,t_1, t_2 \in [0,\infty)$ with $t_0 <t_1< t_2.$ Let $t_*$ be the unique real number for which $t_1= t_* t_0 + (1- t_*)t_2.$ Then we have
\begin{align} \label{def_Ctrualpha2voibien}
\| T\|_{\tilde{\mathcal{C}}^{-t_2}(M)} \le \| T\|_{\tilde{\mathcal{C}}^{-t_1}(M)} \le  C \| T\|^{t_*}_{\tilde{\mathcal{C}}^{-t_0}(M)}  \| T\|_{\tilde{\mathcal{C}}^{-t_2}( M)}^{1- t_*},
\end{align} 
for some constant $C$ independent of $T.$ 
\end{proposition}

 The remaining part of this section is devoted to prove the last proposition. Using a partition of unity as above, that proposition is a direct consequence of Corollary \ref{cor_interpolationspace} at the end of this section. We first recall some notations and results from the interpolation theory of Banach spaces and  refer to \cite{Lunardi,Triebel} for a general treatment of the theory. Then we compute some interpolation spaces of $\tilde{\mathcal{C}}^t(M),$ see Corollary \ref{cor_inter1} below.

Let $A_0$ and $A_1$ be two  Banach spaces which are continuously embedded to a Hausdorff topological vector space $\mathcal{A}.$  Let $B_0$ and $B_1$ be  two  Banach spaces which are continuously embedded to a Hausdorff topological vector space $\mathcal{B}.$ Let $T$ be a  linear operator from $\mathcal{A}$ to $\mathcal{B}.$  Assume that $T|_{A_j}: A_j \rightarrow B_j$ are bounded for $j=0,1.$ The interpolation theory of Banach spaces is to search for Banach subspaces $A \subset \mathcal{A}$ and $B \subset \mathcal{B}$ such that the restriction $T|_{A}: A \rightarrow B$ is a bounded linear operator. The spaces $A$ and $B$ are called interpolation spaces. We will recall below a classical construction of such spaces.   

For $0 <t < \infty$ and $a \in A_0+ A_1,$ define 
\begin{align} \label{def_interpo_0infinity00}
K(t,a; A_0, A_1):= \inf_{a= a_0+ a_1}( \|a_0\|_{A_0}+ t \|a_1\|_{A_1}),  
\end{align}
where $a_0 \in A_0, a_1 \in A_1.$ Let $\alpha$ be a constant in $(0,1).$  The following class of Banach spaces is of great importance in the interpolation theory.

\begin{definition} \label{def_inter} Let  $(A_0,A_1)_{\alpha,\infty}$ be the subspace of $A_0+ A_1$ consisting of $a \in A_0 +A_1$ for which  the following quantity
\begin{align} \label{def_interpo_0infinity}
\|a\|_{(A_0,A_1)_{\alpha,\infty}}:= \sup_{t >0} t^{-\alpha} K(t,a; A_0, A_1)
\end{align}
is finite. The last formula defines a norm on $(A_0,A_1)_{\alpha,\infty}$ which make it to be a Banach space.
\end{definition}
The following fundamental theorem explains the role of the space $(A_0,A_1)_{\alpha, \infty}.$

\begin{theorem} \label{the_interpolationspace} \cite[Th. 1.1.6]{Lunardi} Let $A_0,A_1, B_0, B_1$ and $T$ be as above. Let $\alpha \in (0,1).$ Then the restriction $T_{(A_0,A_1)_{\alpha,\infty}}$ of $T$ to $(A_0,A_1)_{\alpha,\infty}$ is a bounded linear operator from $(A_0,A_1)_{\alpha,\infty}$ to $(B_0,B_1)_{\alpha,\infty}$ and  
$$ \|T|_{(A_0,A_1)_{\alpha,\infty}}\| \le \|T|_{A_0}\|^{1- \alpha} \|T|_{A_0}\|^{\alpha},$$
where $\| \cdot \|$ is the norm of bounded linear operators.
\end{theorem}

Let  $m \in \N^*$ and $k\in \N$ and $\alpha\in (0,1).$  Let $\mathcal{C}^k(\R^m)$ (respectively $\mathcal{C}^{k,\alpha}(\R^m)$)  be the set of $\mathcal{C}^k$ functions (respectively $\mathcal{C}^{k,\alpha}$) on $\R^m.$ For $t \in \R^+,$ define $\mathcal{C}^t(\R^m):= \mathcal{C}^{[t], t-[t]}(\R^m).$  Let $\mathcal{C}^t_b(\R^m)$ be the subset of $\mathcal{C}^t(\R^m)$ consisting of elements whose $\mathcal{C}^t$ norms are bounded.

Let $\Omega$ be a bounded open subset  of $\R^m$ with smooth boundary. Let $\partial \Omega$ be its boundary. Then $\overline{\Omega}$ is a smooth compact manifold with boundary which is itself a global chart. We have the Banach spaces $\mathcal{C}^t(\overline{\Omega})$ and $\tilde{\mathcal{C}}^t(\overline{\Omega})$ as above. In what follows, we will give a description of the interpolation space 
\begin{align} \label{eq_interpola}
\big(\tilde{\mathcal{C}}^{t_0}(\overline{\Omega}),\tilde{\mathcal{C}}^{t_2}(\overline{\Omega})\big)_{\alpha,\infty}
\end{align}
for $0 \le t_0 < t_2< \infty.$ The corresponding interpolation spaces for $\mathcal{C}^{t}(\overline{\Omega})$ and $\mathcal{C}_b^t(\R^m)$ are already known, see  Theorems 2.7.2 and 4.5.2 in \cite{Triebel}.

 It should be noted that the spaces $\big(\mathcal{C}^{t_0}(\overline{\Omega}), \mathcal{C}^{t_2}(\overline{\Omega})\big)_{\alpha,\infty}$ are easily determined by using the result mentioned above for $\mathcal{C}_b^t(\R^m)$ and the fact that the restriction from $\mathcal{C}_b^t(\R^m)$ to $\mathcal{C}^t(\overline{\Omega})$ is a retraction, see \cite[Th.  4.5.1]{Triebel}. Nevertheless, this property is no longer true if we replace $\mathcal{C}^t(\overline{\Omega})$ by $\tilde{\mathcal{C}}^t(\overline{\Omega})$ because even the restriction map from $\mathcal{C}_b^t(\R^m)$ to $\tilde{\mathcal{C}}^t(\overline{\Omega})$ is  not well-defined.  In order to compute (\ref{eq_interpola}), we will follow the original strategy for $\mathcal{C}_b^t(\R^m)$  in  \cite{Triebel}, see also \cite{Lunardi}.  Although, in essence, our below results can be implicitly deduced from \cite{Triebel}, we will present them in a simplified and detailed way which is therefore accessible for a wider audience.    %In particular, for $t>0,$ we have\begin{align} \label{eq_interpola}(\mathcal{C}^0(\R^m), \mathcal{C}^t(\R^m)_{\alpha,\infty}= \mathcal{C}^{\alpha t}(\R^m)\end{align}if $\alpha t$ is not an integer.

 %It might be not hard to derive the results for  $\tilde{\mathcal{C}}^{t}(\overline{\Omega})$ using the general machinery presented in \cite{Triebel}. However, we choose to give below an elementary way, based on the treatment in \cite{Lunardi}, which is accessible for a wider audience.       

The following lemma is well-known but for the reader's convenience, a complete proof will be given. 

\begin{lemma} \label{ex_extend_holderfunction} For every $t \in [0,\infty),$ every  $f \in \mathcal{C}^{t}(\overline{\Omega})$ can be extended to be a function $Ef \in \mathcal{C}^t(\R^m)$ such that $\|Ef \|_{\mathcal{C}^t(\R^m)} \le C \|f\|_{\mathcal{C}^{t}(\overline{\Omega})},$ where $C$ is a constant independent of $f.$ %When $f \in  \tilde{\mathcal{C}}^{t}(\overline{\Omega}),$ the extension $Ef$ can be chosen to be equal to $0$ on $\partial \Omega.$
\end{lemma}

\begin{proof} We will use  a reflexion argument.  By using a partition of unity subordinated to a suitable finite covering of $\overline{\Omega},$ we can suppose that $\overline{\Omega}= \R^{m-1} \times \R^+.$ Let $f \in \mathcal{C}^t(\R^{m-1} \times \R^+).$ Let $[t]$ be the integer part of $t.$ Let $a_1, \cdots, a_{[t]+1}$ be real numbers which are chosen later. Define $Ef :=f$ on $\R^{m-1} \times \R^+$ and  
$$Ef(x_1,\cdots, x_n):= \sum_{k=1}^{[t]+1} a_k f(x_1,\cdots,x_{n-1}, -k x_n)$$
otherwise. It is easy to see that $Ef$ is continuous on $\R^m.$ Now we will choose $a_k$ such that $Ef \in \mathcal{C}^{[t]}.$ If we can do so, we also get $Ef \in \mathcal{C}^t$ because $D^{[t]} Ef$ is $\mathcal{C}^{t- [t]}$ on $\R^{m-1} \times \R^+,$ hence, on the whole $\R^m$ by its defining formula. One only needs to be concerned with the $x_n$-direction. Direct computations show that 
$$\partial_{x_n}^l Ef(x_1,\cdots,x_{n-1},0)= \bigg(\sum_{k=1}^{[t]+1} (-k)^l a_k \bigg) \partial_{x_n}^l f(x_1, \cdots, x_{n-1},0) ,$$
for $0 \le l \le [t].$  The regularity condition on $Ef$ is equivalent to the linear system $\sum_{k=1}^{[t]+1} (-k)^l a_k=1$ for $0 \le l \le [t].$ Its determinant is a Vandermonde one. Hence the system has a unique solution $(a_1,\cdots, a_{[t]+1}).$ When $f|_{\partial \Omega}=0,$ it is clear from the defining formula of $Ef$ that $Ef|_{\partial \Omega}=0.$ The proof is finished.   
\end{proof}

\begin{proposition} \label{pro_inter1} %We have $$\big(\tilde{\mathcal{C}}^{0}(\overline{\Omega}), \tilde{\mathcal{C}}^{1}(\overline{\Omega})  \big)_{\alpha, \infty}= \tilde{\mathcal{C}}^{\alpha}(\overline{\Omega}).$$
 For every $\alpha \in (0,1)$ and every $k \in \N^*,$ we have 
\begin{align} \label{inclusion_Ck}
\big(\tilde{\mathcal{C}}^{0}(\overline{\Omega}), \tilde{\mathcal{C}}^{k}(\overline{\Omega})  \big)_{\alpha, \infty} \supset  \tilde{\mathcal{C}}^{\alpha k}(\overline{\Omega}),
\end{align}
where the last inclusion means a continuous inclusion between Banach spaces.
\end{proposition}

\begin{proof} 
Let $f \in  \tilde{\mathcal{C}}^{\alpha k}(\overline{\Omega}).$ Put $t := \alpha k.$ We write below $\lesssim$ to indicate $\le$ up to a constant independent of $(f,\epsilon).$ By Lemma \ref{ex_extend_holderfunction}, we can extend $f$ to be a function $F$ in $\mathcal{C}^{t}(\R^m)$ with 
$$\|F\|_{\mathcal{C}^{t}(\R^m)} \le C \|f\|_{\tilde{\mathcal{C}}^{t}(\overline{\Omega})},$$
for some constant $C$ independent of $f$ and $F|_{\partial \Omega}=0.$ Let $\B_r$ denotes the ball of radius $r>0$ centered at $0$ in $\R^m$ and $\B_r^+$ denotes the subset of $\B_r$ consisting of $\mathbf{x}=(x_1,\cdots,x_n)$ with $x_n \ge 0.$ Since $\overline{\Omega}$ is compact, we can cover $\partial \Omega$ by a finite number of small open subsets $\{U_j\}_{1\le j \le N}$ of $\R^m$ such that in each  $U_j,$ by a suitable change of coordinates $\Psi_j$, we have 
$$\Psi_j\big(\overline{\Omega} \cap U_j \big)= \B_2^+$$
and $\Psi_j\big(\partial \Omega \cap U_j \big)= \B_2^+ \cap \{x_n \ge 0\}.$ Without loss of generality, we can suppose that $\Psi^{-1}_j(\B^+_{1})$ also covers $\partial \Omega.$ Put 
$$U_0:= \Omega \backslash \cup_{1\le j \le N} \Psi^{-1}_j(\overline{\B}^+_{1}).$$ 
The family $\{U_j\}_{0 \le j \le N}$ covers $\overline{\Omega}.$  Let $\{\chi_j\}_{0 \le j \le N}$ be smooth functions of $\R^m$ such that  $0 \le \chi_j \le 1$ for $0 \le  j \le N,$ and $\supp \chi_j \Subset \Psi^{-1}_j(\B_{5/4})$ for $1 \le  j \le N$ and $\supp \chi_0 \Subset U_0,$ and $\sum_{0 \le j \le N} \chi_j =1$ on $\overline{\Omega}.$

Define $F_j:= (\chi_j F) \circ \Psi^{-1}_j.$ By the properties of $(\Psi_j,F)$ mentioned above, we have $F_j|_{x_n=0}=0.$  Let $\chi$ be a nonnegative smooth  function on $\R^m$ which is compactly supported on $\B_{1}$ such that $\int_{\R^m} \chi d\mathbf{x}=1.$ Taylor's expansion for $F_j$ gives 
\begin{align} \label{eq_taylorF_j}
F_j(\mathbf{x})= F_j(\mathbf{x}-\mathbf{y})+ DF_j(\mathbf{x}-\mathbf{y})\mathbf{y}+ \cdots+ \frac{1}{[t]!} D^{[t]}F_j(\mathbf{x}-\mathbf{y})\mathbf{y}^{[t]}+ R_j(\mathbf{x},\mathbf{y}) \mathbf{y}^{[t]},
\end{align}
where $R_j(\mathbf{x}, \mathbf{y})$ is, for $\mathbf{x}$ fixed, a $\mathcal{C}^{t-[t]}$ linear functional on $ (\R^m)^{[t]}$ and we have 
$$  R_j(\mathbf{x}, 0)= 0, \quad  \| R_j\|_{\mathcal{C}^{t- [t]}} \lesssim \|F_j\|_{\mathcal{C}^{t}} \le C\|f\|_{\mathcal{C}^{t}} .$$
Hence, one gets
\begin{align} \label{ine_uocluongRxy}
|R_j(\mathbf{x},\mathbf{y})| \lesssim | \mathbf{y}|^{t- [t]} \|f\|_{\mathcal{C}^{t}}.
\end{align}
 Put 
$$\epsilon_0:= \min\{1/4, \dist(U_0, \partial \Omega)\}.$$
Let $\epsilon \in (0,\epsilon_0).$ For $0 \le j \le N,$ we define 
\begin{align} \label{def_covoluF}
F_{j,\epsilon}(\mathbf{x}):= \int_{\R^m} \big[F_j(\mathbf{x}- \epsilon \mathbf{y})+ DF_j(\mathbf{x}-\epsilon\mathbf{y})(\epsilon\mathbf{y})+ \cdots + \frac{1}{[t]!} D^{[t]}F_j(\mathbf{x}- \epsilon\mathbf{y})(\epsilon\mathbf{y})^{[t]}  \big] \chi(\mathbf{y}) d\mathbf{y}.
\end{align}
Observe that $F_{0,\epsilon}$ is a smooth function in  $\tilde{\mathcal{C}}^{\infty}(\overline{\Omega}) $ by the choice of $\epsilon$ and $F_{j,\epsilon}$ is smooth on $\R^m$ and compactly supported on $\B_{3/2}$ for $1 \le j \le N.$ A property of the convolution implies that $F_{j,\epsilon}$ converges to $F_j$ in $\mathcal{C}^0$-topology. Precisely, using (\ref{eq_taylorF_j}),  (\ref{def_covoluF}) and (\ref{ine_uocluongRxy}) yields that 
\begin{align} \label{ine_hieuFjvaFjepsilon}
|F_{j,\epsilon}(\mathbf{x})- F_{j}(\mathbf{x})| \le    \epsilon^{[t]}\int_{\R^m} |R_j(\mathbf{x}, \epsilon\mathbf{y})| \chi(\mathbf{y}) d\mathbf{y} \le C \epsilon^{t}\|f\|_{\mathcal{C}^{t}},
\end{align}
for every $\mathbf{x}.$ Let $\tau$ be a smooth function on $\R$ compactly supported on $[-2,2]$ such that $\tau \equiv 1$ on $[-3/2,3/2].$   Define 
$$F'_{j,\epsilon}(x_1, \cdots,x_n):= F_{j,\epsilon}(x_1,\cdots,x_{n-1},x_n)-  \tau(x_n) F_{j,\epsilon}(x_1,\cdots,x_{n-1},0),$$
for $1 \le j \le N$ and we put $F'_{0,\epsilon}:= F_{0,\epsilon}$ for consistence. We immediately see that $F'_{j,\epsilon}= 0$ on $\{x_n=0\}$ and $\supp F'_{j,\epsilon} \subset \B_2.$  As a consequence,  $F'_{j,\epsilon} \circ \Psi_j$ is smooth on $\R^m$ and vanishes on  $\partial \Omega.$ We deduce from (\ref{ine_hieuFjvaFjepsilon}) and the fact that $F_j|_{\{x_n=0\}} \equiv 0$ that 
\begin{multline} \label{ine_hieuFjvaFjepsilon2phay}
|F'_{j,\epsilon}(\mathbf{x})- F_{j}(\mathbf{x})| \le |F_{j,\epsilon}(\mathbf{x})- F_{j}(\mathbf{x})|+\\
 |F_{j,\epsilon}(x_1,\cdots,x_{n-1},0)- F_{j}(x_1,\cdots,x_{n-1},0)| \le 2C\epsilon^{t} \|f\|_{\mathcal{C}^{t}}. 
\end{multline}
Define 
$$g_{1,\epsilon}:= \sum_{0 \le j \le N} F'_{j, \epsilon} \circ \Psi_j|_{\overline{\Omega}} \in \tilde{\mathcal{C}}^{\infty} (\overline{\Omega})$$
and $g_{0,\epsilon}:= f - g_{1,\epsilon} \in \tilde{\mathcal{C}}^0(\overline{\Omega}).$ We have $f= g_{0,\epsilon}+ g_{1,\epsilon}.$ In view of (\ref{def_interpo_0infinity}), we have to estimate $\|g_{0,\epsilon}\|_{ \tilde{\mathcal{C}}^{0}(\overline{\Omega})}$ and $\|g_{1,\epsilon}\|_{ \tilde{\mathcal{C}}^{k}(\overline{\Omega})}.$ Since $f= \sum_{0 \le j \le N} F_j \circ \Psi_j$, we have 
$$g_{0,\epsilon}= \sum_{0 \le j \le N } ( F_j  \circ \Psi_j - F'_{j,\epsilon} \circ \Psi_j).$$
Taking into account (\ref{ine_hieuFjvaFjepsilon2phay}), one gets 
\begin{align} \label{ine_danigiag_0epsilon}
\|g_{0,\epsilon}\|_{ \tilde{\mathcal{C}}^{0}(\overline{\Omega})} \lesssim  \epsilon^{t} \|f\|_{\tilde{\mathcal{C}}^{t}(\overline{\Omega})}. 
\end{align}
For $0 \le l \le [t],$ we  define 
$$G_{j,l}(\mathbf{x}, \mathbf{y}):= D^l F_j(\mathbf{y})+D^{l+1} F_j(\mathbf{y})\mathbf{y}+  \cdots + \frac{1}{([t]-l)!} D^{[t]}F_j(\mathbf{y})(\mathbf{x}-\mathbf{y})^{[t]-l}$$
which is the Taylor expansion up to the $([t]-l)$ order of $D^l F_j(\mathbf{x})$ at $\mathbf{y}.$ Thus arguing as in (\ref{ine_uocluongRxy}), we get 
\begin{align} \label{ine_GjlDlFj}
|G_{j,l}(\mathbf{x}, \mathbf{y}) - D^l F_j(\mathbf{x})| \lesssim  \|f\|_{\mathcal{C}^t} |\mathbf{y}|^{t-l}. 
\end{align}
The equality (\ref{def_covoluF}) can be rewritten as
$$F_{j,\epsilon}(\mathbf{x})= \epsilon^{-m}\int_{\R^m}\big[F_j(\mathbf{y}')+ DF_j(\mathbf{y}')(\mathbf{x}-\mathbf{y}')+ \cdots + \frac{1}{[t]!} D^{[t]}F_j(\mathbf{y}')(\mathbf{x}-\mathbf{y}')^{[t]}  \big]  \chi(\frac{\mathbf{x}-\mathbf{y}'}{\epsilon}) d\mathbf{y}'.$$
Differentiating the last equality in $\mathbf{x}$ for $k'$ times gives
\begin{align} \label{ine_uoncluongDkcuaFjepsilon}
D^{k'}_{\mathbf{x}}F_{j,\epsilon}(\mathbf{x}) &= \epsilon^{-m-k'+ l} \sum_{0 \le l \le \min\{k', [t]\}} \int_{\R^m} G_{j,l}(\mathbf{x},\mathbf{y}') \otimes D^{k'-l}\chi(\frac{\mathbf{x}-\mathbf{y}'}{\epsilon}) d\mathbf{y}' \\
\nonumber
&=\epsilon^{-k'+l } \sum_{0 \le l \le \min\{k', [t] \}} \int_{\R^m} G_{j,l}(\mathbf{x},\mathbf{x}- \epsilon\mathbf{y}) \otimes  D^{k'-l}\chi(\mathbf{y}) d\mathbf{y}
\end{align}
by a suitable change of coordinates.  Since $\int_{\R^m}D^l_{\mathbf{x}}\chi(\mathbf{y}) d\mathbf{y}=0$ for any $l \ge 1,$  we obtain 
\begin{align} \label{ine_tichphanGjl}
\int_{\R^m} G_{j,l}(\mathbf{x},\mathbf{x}- \epsilon\mathbf{y}) \otimes  D^{k'-l}\chi(\mathbf{y}) d\mathbf{y}= \int_{\R^m} \big(G_{j,l}(\mathbf{x},\mathbf{x}- \epsilon\mathbf{y})- D^l F_j(\mathbf{x}) \big) \otimes  D^{k'-l}\chi(\mathbf{y}) d\mathbf{y}
\end{align}
which is of absolute value $\lesssim \epsilon^{t-l} \| f\|_{\mathcal{C}^t}$ by using (\ref{ine_GjlDlFj}) and the fact that $\supp \chi \subset \B_1.$  Combining (\ref{ine_uoncluongDkcuaFjepsilon}) with (\ref{ine_tichphanGjl}) gives 
$$|D^{k'}_{\mathbf{x}}F_{j,\epsilon}(\mathbf{x})| \lesssim \epsilon^{-k'+t} \| f\|_{\mathcal{C}^t}$$
which implies that
\begin{align} \label{ine_danigiag_0epsilon2}
\|g_{1,\epsilon}\|_{ \tilde{\mathcal{C}}^{k}(\overline{\Omega})} \lesssim   \epsilon^{-k+t} \| f\|_{\mathcal{C}^t}
\end{align}
by choosing $k'=k.$ Taking into account (\ref{ine_danigiag_0epsilon}), (\ref{ine_danigiag_0epsilon2}) and  (\ref{def_interpo_0infinity00}), one deduces that
$$\epsilon^{-\alpha k} K\big(\epsilon^k, f;\tilde{\mathcal{C}}^{0}(\overline{\Omega}), \tilde{\mathcal{C}}^{k}(\overline{\Omega}) \big) \le \epsilon^t \big( \|g_{0,\epsilon}\|_{\tilde{\mathcal{C}}^0}+ \epsilon^k\|g_{1,\epsilon}\|_{\tilde{\mathcal{C}}^k} \big)  \lesssim  \|f\|_{\tilde{\mathcal{C}}^{t}(\overline{\Omega})},$$
for every $\epsilon \in (0,\epsilon_0).$ When $\epsilon \ge \epsilon_0,$ since 
$$f= f+0 \in \tilde{\mathcal{C}}^{0}(\overline{\Omega})+\tilde{\mathcal{C}}^{1}(\overline{\Omega}),$$
we have 
$$\epsilon^{-\alpha k} K\big(\epsilon^k, f;\tilde{\mathcal{C}}^{0}(\overline{\Omega}), \tilde{\mathcal{C}}^{k}(\overline{\Omega}) \big) \le \epsilon_0^{-\alpha k} \|f\|_{\tilde{\mathcal{C}}^{0}(\overline{\Omega})} \le \epsilon_0^{-\alpha k} \|f\|_{\tilde{\mathcal{C}}^{\alpha k }(\overline{\Omega})}.$$
Hence, $f \in \big(\tilde{\mathcal{C}}^{0}(\overline{\Omega}), \tilde{\mathcal{C}}^{k}(\overline{\Omega}) \big)_{\alpha, \infty}.$ The proof is finished.
\end{proof}

For every $h \in \R^m$ and every a function $g$ on $\R^m,$ define the operator 
$$\Delta_h g(x):= g(x+h)- g(x)$$
for every $x \in \R^m.$ The following property is crucial for the next proposition.

\begin{lemma}  \label{le_holder_norm} Let $\alpha \in (0,1)$ and $l$ be an integer $\ge 1.$ For $g \in \mathcal{C}^{\alpha}_b(\R^m),$ we put 
$$\|g\|_{\alpha,\Delta,l}:= \|g\|_{\mathcal{C}^0}+ \sup_{x,h \in \R^m, h \not =0} \frac{|\Delta_h^l g |}{|h|^{\alpha}} \cdot$$ 
Then the last formula defines a norm on $\mathcal{C}^{\alpha}_b(\R^m)$ which is equivalent to its usual $\mathcal{C}^{\alpha}$ norm.  More precisely, there exists a positive constant $C_{l,\alpha}$ depending only on $(l,\alpha)$ such that  for every $g,$ we have 
$$C_{l,\alpha}^{-1}\|g\|_{\mathcal{C}^{\alpha}} \le  \|g\|_{\alpha,\Delta,l} \le C_{l,\alpha} \|g\|_{\mathcal{C}^{\alpha}}.$$ 
\end{lemma}

\begin{proof} This is a simplification of Lemma 1.13.4 in \cite{Triebel}. When $l=1,$ the two norms are identical. Consider $l\ge 2.$  Observe that it is enough to prove the desired result for $l=2$ because the general case can easily follow by induction. It is clear that $\|g\|_{\alpha,\Delta,2} \le 2 \|g\|_{\mathcal{C}^{\alpha}}.$ We now prove the converse inequality. The key argument is the following formula:
$$g(x+h)- g(x)= \frac{1}{2}\big( g(x+2h)- g(x) \big) - \frac{g(x+2h)- 2 g(x+h)+ g(x)}{2}.$$ 
Dividing the last equality by $|h|^{\alpha}$ gives 
$$\frac{|g(x+h) - g(x)|}{|h|^{\alpha}} \le 2^{\alpha - 1} \frac{|g(x+2h) - g(x)|}{|2h|^{\alpha}}+  \frac{|g(x+2h)- 2 g(x+h)+ g(x)|}{2 |h|^{\alpha}}.$$
By taking the supremum over $\{(x,h) \in \R^{2m}, h \not =0\}$ in the last inequality, we deduce that
$$\|g\|_{\mathcal{C}^{\alpha}} \le 2^{\alpha -1} \|g\|_{\mathcal{C}^{\alpha}}+ \|g\|_{\alpha,\Delta,2}.$$ 
Since $2^{\alpha-1} <1$ we get the desired conclusion. The proof is finished.
\end{proof}

\begin{proposition} \label{pro_interbeta0} Let $k$ be a positive integer and let  $\alpha$ be a real number in $(0,1).$ Assume that $\alpha k \in (0,1).$ Then we have
 \begin{align} \label{inclusion_Ckalpha01}
\big(\tilde{\mathcal{C}}^{0}(\overline{\Omega}), \tilde{\mathcal{C}}^{k}(\overline{\Omega})  \big)_{\alpha, \infty} \subset \tilde{\mathcal{C}}^{\alpha k}(\overline{\Omega}).
\end{align}
\end{proposition}

\begin{proof} %By Definition \ref{def_inter}, we have  $$ \big(\tilde{\mathcal{C}}^{0}(\overline{\Omega}), \tilde{\mathcal{C}}^{1}(\overline{\Omega})  \big)_{\alpha, \infty} \subset (\mathcal{C}^0(\R^m), \mathcal{C}^t(\R^m)_{\alpha,\infty}.$$ Therefore, using (\ref{eq_interpola}), we see that 

Let take an element $f \in \big(\tilde{\mathcal{C}}^{0}(\overline{\Omega}), \tilde{\mathcal{C}}^{k}(\overline{\Omega})  \big)_{\alpha, \infty}.$   Suppose that $f=g_0 + g_1$ with $g_0 \in \tilde{\mathcal{C}}^{0}(\overline{\Omega})$ and $g_1 \in \tilde{\mathcal{C}}^{k}(\overline{\Omega}).$ We have $\Delta_h^k f=\Delta_h^k g_0+\Delta_h^k g_1.$ By using Taylor's expansion of $g_1,$ observe that $|\Delta_h^k g_1| \le C |h|^k \|g_1\|_{\mathcal{C}^k}$ for some constant $C$ independent of $(g_1,h).$  On the other hand,  $|\Delta_h^k g_0| \le 2^l \|g_0\|_{\mathcal{C}^0}.$ Combining these inequalities gives 
 $$|\Delta_h^k f| \le 2^l \|g\|_{\mathcal{C}^0}+ C  |h|^k \|g_1\|_{\mathcal{C}^k} \lesssim \|g\|_{\mathcal{C}^0}+  |h|^k \|g_1\|_{\mathcal{C}^k},$$
for every $(g_0, g_1)$ with $f= g_0+ g_1.$ Taking the infimum in the last inequality in $(g_0, g_1),$ we obtain 
$$|\Delta_h^k f|  \lesssim  K\big(h^k, f;\tilde{\mathcal{C}}^{0}(\overline{\Omega}), \tilde{\mathcal{C}}^{1}(\overline{\Omega}) \big) \le |h|^{\alpha k} \|f\|_{\big(\tilde{\mathcal{C}}^{0}(\overline{\Omega}), \tilde{\mathcal{C}}^{k}(\overline{\Omega}) \big)_{\alpha, \infty} }.$$
As a consequence, one gets 
$$\|f\|_{\alpha k, \Delta,k} \lesssim  \|f\|_{\big(\tilde{\mathcal{C}}^{0}(\overline{\Omega}), \tilde{\mathcal{C}}^{k}(\overline{\Omega}) \big)_{\alpha, \infty} }.$$
By Lemma \ref{le_holder_norm} and the hypothesis that $\alpha k <1,$ we obtain the desired result. The proof is finished.
\end{proof}

\begin{corollary} \label{cor_inter1} %We have $$\big(\tilde{\mathcal{C}}^{0}(\overline{\Omega}), \tilde{\mathcal{C}}^{1}(\overline{\Omega})  \big)_{\alpha, \infty}= \tilde{\mathcal{C}}^{\alpha}(\overline{\Omega}).$$
%For every $k \in \N^*$ and $\beta \in [0,1),$ we have 
%\begin{align} \label{inclusion_Ckalpha}
%\big(\tilde{\mathcal{C}}^{\beta}(\overline{\Omega}), \tilde{\mathcal{C}}^{k}(\overline{\Omega})  \big)_{\alpha, \infty}=  \tilde{\mathcal{C}}^{\alpha k+ (1- \alpha) \beta}(\overline{\Omega}).
%\end{align}
For every $\alpha \in (0,1),$ every real nonnegative numbers $t_1$ and $t_2,$ we have 
\begin{align} \label{inclusion_Ckalpha}
\big(\tilde{\mathcal{C}}^{t_1}(\overline{\Omega}), \tilde{\mathcal{C}}^{t_2}(\overline{\Omega})  \big)_{\alpha, \infty} \supset \tilde{\mathcal{C}}^{\alpha t_2+ (1- \alpha)t_1}(\overline{\Omega}).
\end{align}
\end{corollary}

\begin{proof} For simplicity, we give a proof for $ t_2= k \in \N^*$ and $t_1= \beta \in [0,1).$ The general case can be deduced by using similar arguments.  By a consequence of the reiteration theorem (see \cite[Re. 1.3.7]{Lunardi}), we have the following general formula: 
$$\big((A_0, A_1)_{\theta, \infty}, A_1\big)_{\alpha, \infty}= ( A_0,A_1)_{(1-\alpha)\theta+\alpha, \infty}.$$
Applying the last equality to 
$$A_0=\tilde{\mathcal{C}}^{0}(\overline{\Omega}), A_1= \tilde{\mathcal{C}}^{k}(\overline{\Omega}) \quad  \text{and} \quad \theta= \beta/k$$
and using the fact that $(A_0, A_1)_{\theta, \infty}= \tilde{\mathcal{C}}^{\beta}(\overline{\Omega})$ (by Proposition \ref{pro_interbeta0} and \ref{pro_inter1}), we obtain the desired inclusion. The proof is finished.
\end{proof}

Since $(\R, \R)_{\alpha,\infty}= \R$ for any $\alpha \in (0,1),$ applying Theorem \ref{the_interpolationspace} to 
$$A_0=\tilde{\mathcal{C}}^{t_0}(\overline{\Omega}), \quad A_1=\tilde{\mathcal{C}}^{ t_2}(\overline{\Omega}), \quad B_0= B_1=\R,$$
 and then using Corollary \ref{cor_inter1}, we obtain the following result.

\begin{corollary} \label{cor_interpolationspace} Let $\Omega$ be a bounded open subset of $\R^m$ with smooth boundary. Let $t_0, t_1$ and $t_2$ be three real numbers such that $0 \le t_0 < t_1 < t_2.$  Let $S$ be a bounded linear map from $\tilde{\mathcal{C}}^{t_0}(\overline{\Omega})$ to $\R.$ Then the restriction $S|_{\tilde{\mathcal{C}}^{t_j}(\overline{\Omega})}$ of $S$ to $\tilde{\mathcal{C}}^{t_j}(\overline{\Omega})$ for $j=1$ or $2$ is also a bounded linear map from $\tilde{\mathcal{C}}^{t_j}(\overline{\Omega})$ to $\R$ and % Then the restriction $S|_{\tilde{\mathcal{C}}^{\alpha k}(\overline{\Omega})}: \tilde{\mathcal{C}}^{\alpha k}(\overline{\Omega}) \rightarrow \R$ is bounded and 
$$ \|S|_{\tilde{\mathcal{C}}^{t_1}(\overline{\Omega})}\|  \le  c \|S|_{\tilde{\mathcal{C}}^{t_0}(\overline{\Omega})}\|^{t_*}  \|S|_{\tilde{\mathcal{C}}^{ t_2}(\overline{\Omega})}\|^{1- t_*},$$
where $c$ is a constant independent of $S$ and $t_*$ is the unique real number for which $t_1= t_* t_0 + (1-t_*) t_2.$
\end{corollary}

%\begin{corollary} Let $k$ be a positive integer. Let $\alpha \in (0,1).$ Then we have $$\big(\tilde{\mathcal{C}}^{0}(\Omega), \tilde{\mathcal{C}}^{k}(\Omega)  \big)_{\frac{\alpha}{k}, \infty}= \tilde{\mathcal{C}}^{\alpha}(\Omega).$$\end{corollary}

\section{Analytic discs partly attached to a generic submanifold} \label{sec_analyticdisc}

Firstly we fix some notations which will be valid throughout the rest of paper. For every Riemannian smooth manifold $Y,$ any $ a\in Y$ and $r \in \R^+,$ we denote by $\B_Y(a,r)$ the ball of radius $r$ centered at $a$ of $Y$ and by $vol_Y$ the Riemannian volume form of $Y$. When $Y= \R^m$ for some $m \in \N$ with the Euclidean metric, we write $\B_m(a,r)$ instead of $\B_Y(a,r)$ and $\B_m$ instead of $\B_m(0,1).$ In particular, when $Y=\C  \simeq   \R^2$ and $a=0,$  we put $\D_r:= \B_2(0,r)$ and $\D:= \B_2(0,1).$ For every $m \in\N^*,$  we identify $\C^m$ with $\R^{2m}$ via the formula $\C^m= \R^m+ i \R^m.$

 Let $\partial \D$ be the boundary of $\D$ and $\partial^+ \D:= \{\xi \in \D: \Re \xi \ge 0\}.$  We sometimes identify $\xi \in \D$ with $\theta \in (-\pi, \pi]$ by letting $\xi= e^{i\theta}.$ \emph{An analytic disc} $f$ in $X$ is a holomorphic mapping from $\D$ to $X$ which is continuous up to the boundary $\partial \D$ of $\D.$ For an interval $I \subset \partial \D,$ $f$ \emph{is said to be $I$-attached to a subset $E \subset X$} if $f(I) \subset E.$ When $I= \partial^+ \D,$ an analytic disc $I$-attached to $E$ is said to be \emph{half-attached to $E.$}  %Fix a Riemannian metric on $X.$ For $p\in X$ and $r>0,$ let $B_X(p,r)$ be the ball centered at $p$ of radius $r$ of $X.$ Put  $B^*_X(p,r):= B_X(p,r)\backslash \{p\}.$ 

Let $K$ be a generic immersed  $\mathcal{C}^3$ submanifold of $X.$ Observe that the dimension of $K$ is at least $n.$ \emph{Throughout the paper, we  only consider the case where $\dim K=n$, hence its codimension $d$ equals $n$.} This is in fact the most interesting case and the general case will be easily deduced from it. In Section \ref{sec_superpotential}, we will explain the necessary modifications to get Theorem \ref{the_MAgeneric2} when $\dim K >n.$

Our goal is to for each $a \in K$ construct a $\mathcal{C}^{2,1/2}$-differentiable family of analytic discs partly attached to $K$ which covers an open neighborhood of $a$ in $X.$ It should be noted that any family of discs partly attached to $K$ degenerates near $K$ due to its attachment to $K.$ Controlling such behaviour around $K$ is actually the key point in this section.  We also need that the part of this family lying in $K$ must cover an open neighborhood of $0$ in $K.$ Constructing analytic discs is an important tool in Cauchy-Riemann geometry. Generally, one uses a suitable Bishop-type equation together with a choice of initial data depending on situations to obtain the desired result. The reader may also consult \cite{Baouendi_Ebenfelt_Rothschild,MerkerPorten,MerkerPorten2} and references therein for more information. In what follows, we will apply the same strategy combining with the ideas from \cite{Vu_feketepoint}. 

The following local coordinates are frequently used in the Cauchy-Riemann geometry.
 
\begin{lemma} \label{pro_localcoordinates} Through every point $a$ of $K,$ there exist local holomorphic coordinates $(W, \mathbf{z})$ of $X$ around $a$ such that in that local coordinates,  the point $a$ is the origin and  $K\cap W$ is the graph over $\B_n$ of a $\mathcal{C}^3$ map $h$ from $\overline{\B}_n$ to $\R^n$ which satisfies $D^j h(0)=0$ with $j=0,1,2,$ where $Dh$ denotes the differential of $h.$ Moreover, $\|h\|_{\mathcal{C}^{3}}$ is bounded uniformly in $a \in \tilde{K}.$  
\end{lemma}
 
\begin{proof} The existence of such $h$ with $h(0)=Dh(0)=0$ is well-known, see \cite{Baouendi_Ebenfelt_Rothschild} for example. In order to obtain the additional property $D^2 h(0)=0,$ one will need to perform a change of coordinates, we refer to \cite[Sec. 6.10]{MerkerPorten} for details. The proof is finished.
\end{proof}

From now on, fix an arbitrary point $a \in K$ and we confine ourselves to the local chart described in Lemma \ref{pro_localcoordinates}. In other words, we will work on $\C^n$ and
$$K':= \{\mathbf{z}=\mathbf{x}+i h(\mathbf{x} ) \in \C^n :  \mathbf{x} \in \B_n\},$$
where we have $h(0)= Dh(0)=0.$ For most of the time, the last condition is enough for our purposes, we will only need $D^2h(0)=0$ in the proof of  Proposition \ref{pro_volumepshonRnddcnew}.  The property of $h$ yields that there is a constant $c_0$ for which
\begin{align} \label{ine_chuanC3cuahdanhgialuythua3} 
|h(\mathbf{x})| \le c_0 |\mathbf{x}|^2, \quad |Dh(\mathbf{x})| \le c_0 |\mathbf{x}|, 
\end{align}
for every $\mathbf{x} \in \overline{\B}_n.$ 

In this paragraph, we prepare some useful facts about harmonic functions on the unit disc which will be indispensable for studying Bishop-type equations later.   Denote by $z=x+ i y$ the complex variable on $\C$ and by $\xi=e^{i\theta}$ the variable on $\partial \D.$ Let $u_0(\xi)$ be an arbitrary continuous function on $\partial \D.$  Recall that $u_0$ can be extended uniquely to be a harmonic function on $\D$ which is continuous on $\overline{\D}$. Since this correspondence is bijective, without stating explicitly, we will freely identify $u_0$ with its harmonic extension on $\D.$ We will write $u_0(z)=u_0(x+iy)$ to indicate the harmonic extension of $u_0(e^{i\theta}).$ %By \cite[Th. 6.19]{Trudinger_Gilbarg}, 
It is well-known  that the Cauchy transform of $u_0,$ given by 
$$\mathcal{C}u_0(z):= \frac{1}{2\pi} \int_{-\pi}^{\pi} u_0(e^{i\theta}) \frac{e^{i\theta}+ z}{e^{i\theta}-z} d\theta,$$
is a holomorphic function on $\D$ whose real part is $u_0.$ Let $\mathcal{T}u_0$ be the imaginary part of $\mathcal{C}u_0.$ Decomposing the last formula into the real and imaginary parts, we obtain
\begin{align} \label{equ_harmonicextension}
u_0(z)= \frac{1}{2 \pi} \int_{-\pi}^{\pi} \frac{(1- |z|^2)}{|e^{i\theta}-z|^2} u_0(e^{i\theta}) d\theta.
\end{align}
and 
\begin{align*} %\label{equ_harmonicextensionTu}
\mathcal{T}u_0(z)= \frac{1}{2 \pi } \int_{-\pi}^{\pi } \frac{(z e^{- i\theta}- \bar{z} e^{i\theta})}{i |e^{i\theta}-z|^2} u_0(e^{i\theta}) d\theta.
\end{align*}
The function $\mathcal{T} u_0$ is harmonic on $\D$ but is not always continuous up to the boundary of $\D.$ Let $k$ be an arbitrary natural number and let  $\beta$ be an arbitrary number in $(0,1).$ A result of Privalov (see \cite[Th. 4.12]{MerkerPorten}) implies that if $u_0$ belongs to  $\mathcal{C}^{k,\beta}(\partial\D)$, then $\mathcal{T}u_0$ is continuous up to $\partial \D$ and  $\|\mathcal{T}u_0\|_{\mathcal{C}^{k,\beta}(\partial \D)}$  is bounded by  $ \|u_0\|_{\mathcal{C}^{k,\beta}(\partial \D)}$ times a constant independent of $u_0.$ Hence, the linear self-operator of $\mathcal{C}^{k, \beta}(\partial \D)$ defined by sending $u_0$ to the  restriction of $\mathcal{T}u_0$ onto $\partial \D$ is bounded and called  \emph{the Hilbert transform}. For simplicity, we also denote it by $\mathcal{T}.$   For our later purposes,  it is convenient to use a modified version $\mathcal{T}_1$ of $\mathcal{T}$ defined by
 $$\mathcal{T}_1 u_0:= \mathcal{T}u_0 - \mathcal{T}u_0(1).$$
Hence we  always have $\mathcal{T}_1 u_0(1)=0$ and
\begin{align}\label{eq_daohamcuaT}
\partial_{\theta}\mathcal{T}_1 u_0=\partial_{\theta} \mathcal{T} u_0= \mathcal{T}\partial_{\theta} u_0,
\end{align}
provided that  $u_0 \in \mathcal{C}^{1, \beta}(\partial \D)$ with $\beta \in (0,1),$  see \cite[p.121]{MerkerPorten} for a proof.  The boundedness of $\mathcal{T}$ on $\mathcal{C}^{k, \beta}(\partial \D)$  implies that there is a constant $C_{k,\beta}>1$ such that for any $v \in \mathcal{C}^{k,\beta}(\partial \D)$ we have  
\begin{align} \label{ine_chuancuaT}
\|\mathcal{T}_1 v\|_{\mathcal{C}^{k,\beta}(\partial \D)} \le C_{k,\beta} \|v\|_{\mathcal{C}^{k,\beta}(\partial \D)}.
\end{align}
Extending $u_0, \mathcal{T}_1 u_0$ harmonically to $\D.$ By construction, the function $f(z):= -\mathcal{T}_1 u_0(z) + i u_0(z)$ is holomorphic on  $\D$ and continuous on $\overline{\D}$ provided that $u_0$ is in $\mathcal{C}^{\beta}(\partial \D)$ with $0< \beta <1.$ By \cite[Th. 4.2]{MerkerPorten2}, $\|f\|_{\mathcal{C}^{k,\beta}(\overline{\D})}$ is bounded by $ \|f\|_{\mathcal{C}^{k,\beta}(\partial \D)}$ times a constant depending only on $(k,\beta).$
Since $\| u_0\|_{\mathcal{C}^{k,\beta}(\overline{\D})} \le \|f\|_{\mathcal{C}^{k,\beta}(\overline{\D})}$ and $\|f\|_{\mathcal{C}^{k,\beta}(\partial\D)}\le  (1+ C_{k,\beta}) \|u_0\|_{\mathcal{C}^{k,\beta}(\partial \D)}$ by (\ref{ine_chuancuaT}), we have
\begin{align} \label{ine_danhgiachuaCkcuauvoibien}
\| u_0\|_{\mathcal{C}^{k,\beta}(\overline{\D})} \le C'_{k,\beta} \| u_0\|_{\mathcal{C}^{k,\beta}(\partial \D)}, 
\end{align}
for  some constant $C'_{k,\beta}$ depending only on $(k,\beta).$ A direct consequence of the above inequalities is that when $u_0$ is smooth on $\partial \D$, the associated holomorphic function $f$ is also smooth on $\overline{\D}.$

\begin{lemma} \label{le_version8existenceu_MA} There exist a function  $u_0 \in \mathcal{C}^{\infty}(\partial \D)$ and two positive constants $(\theta_{u_0}, c_{u_0})$ such that  $u_0(e^{i\theta}) = 0$ for  $\theta \in [-\theta_{u_0}, \theta_{u_0}] \subset [-\pi/2, \pi/2]$ and  $\partial_x u_0(1)=-1$ and $u_0(z) >c_{u_0}(1- |z|)$ for every $z \in \D.$ %and for some constant  $c_0$ independent of $z???$
 \end{lemma}

\begin{proof} Let $u$ be a smooth function on $\partial \D$ vanishing on $\partial^+ \D.$ By Poisson's formula, we have
\begin{align} \label{equ_harmonicextension_MA}
u(z)= \frac{1}{2 \pi} \int_{-\pi}^{\pi} \frac{(1- |z|^2)}{|e^{i\theta}-z|^2} u(e^{i\theta}) d\theta.
\end{align}
Differentiating (\ref{equ_harmonicextension_MA}) gives 
$$\partial_x u(1)=\frac{1}{2\pi} \int_{-\pi}^{\pi} \frac{u(e^{i\theta})}{\cos\theta -1}d\theta.$$
Note that the last integral is well-defined because $u$ vanishes on $\partial^+ \D.$ It is easy to choose a smooth $u$ so that the above integral is equal to $-1$ and $u \equiv 0$ on $\partial^+ \D$ and $u(e^{i\theta})>0$ for $|\theta| \ge 3\pi/2.$ The last property and (\ref{equ_harmonicextension_MA}) show that $u(z)>0$ for every $z \in \D.$

We have chosen $u$ with the property that $\partial_x u(1)=-1$ and $u(z) >0$ for $z \in \D.$ This implies that $\partial_x u(e^{i\theta}) \le -1/2$ for every $\theta \in [-\theta_0, \theta_0] \subset (-\pi/2, \pi/2)$ for $\theta_0$ close enough to $1.$ Since $u$ vanishes on $\partial^+ \D,$ we have 
$$0 = \partial_{\theta} u(e^{i\theta})=- \partial_x u(e^{i\theta}) \sin \theta+ \partial_y u(e^{i\theta}) \cos \theta$$
which yields 
\begin{align} \label{eq_partialxyu}
 \partial_y u(e^{i\theta})=  \partial_x u(e^{i\theta}) \tan \theta
\end{align}
for $\theta \in [-\theta_0, \theta_0].$  Let $z= |z| e^{i\theta} \in \D$ such that $\theta \in [-\theta_0, \theta_0].$ Taylor's expansion for $u$ at $e^{i\theta}$ gives 
\begin{align} \label{eq_Taylorhuongru}
u(|z| e^{i\theta}) &= u(e^{i\theta})+ (|z| \cos \theta - \cos \theta ) \partial_x u(e^{i\theta})+  (|z| \sin \theta - \sin \theta ) \partial_y u(e^{i\theta})+ O\big((1- |z|)^2\big)\\
\nonumber
&=  \frac{(|z|-1) \partial_x u(e^{i\theta})}{\cos \theta}+O\big((1- |z|)^2\big)    \quad  \text{(by (\ref{eq_partialxyu}))}.
\end{align}
By our choice of $\theta_0,$ the last equality gives  
\begin{align} \label{eq_partialxyugiatricuau}
 u(|z|e^{i\theta}) \ge   (1- |z|)/2 - \|u\|_{\mathcal{C}^2(\D)}(1 - |z|)^2 \ge (1 -|z|)/4,
\end{align}
for $|z| \ge 1- 1/4 \|u\|^{-1}_{\mathcal{C}^2(\D)} .$ When $|z| \le 1- 1/4 \|u\|^{-1}_{\mathcal{C}^2(\D)},$ we have $u(z) \ge c$ for some constant $c$ independent of $z.$ This combined with the fact that $(1- |z|) \le 1$ implies that there is a positive constant $c'$ for which $u(z) \ge c' (1- |z|)$ for $|z| \le 1- 1/4 \|u\|^{-1}_{\mathcal{C}^2(\D)}.$ In summary, we can find a positive constant $c'$ for which 
$$u(z) \ge c' (1 -|z|),$$ 
for $z= |z| e^{i\theta} \in \D$ with $\theta \in [-\theta_0, \theta_0].$  

 Now let $\Omega$ be a simply connected subdomain of $\D$ with smooth boundary such that $\Omega$ is strictly convex and $\overline{\Omega} \cap \overline{\D}= [e^{-i \theta_0/2}, e^{i\theta_0/2}].$ %and for any $z \in \Omega$ close to $[e^{-i \theta_0/2}, e^{i\theta_0/2}],$   the line $Oz$ passing through $z$ and the origin $O$ is transverse to $\partial \Omega.$ 
By Painvel\'e's theorem (see, for example, \cite[Th. 3.1]{Bell} or \cite[Th. 5.3.8]{Krantz_geo_func}), there is a smooth diffeomorphism $\Phi$ from $\overline{\D}$ to $\overline{\Omega}$ which is a biholomorphism from $\D$ to $\Omega$ and $\Phi(1)=1$. Define  $u'_0:= u \circ \Phi$ which is a smooth function on $\overline{\D}$ and  harmonic on $\D.$ We immediately have $u'_0(z) >0$ on $\D.$ 

Since $\Phi(1)=1$ and $\Phi$ sends $\partial \D$ to $\partial \Omega,$ there is a small positive constant $\theta'$ such that $\Phi([e^{- i \theta'_0}, e^{i \theta'_0}])$ is contained in $[e^{-i \theta_0/2}, e^{i\theta_0/2}].$ This yields  $u'_0(e^{i\theta})= 0$ for  $ |\theta| \le \theta_0'$ and $\Re^2 \Phi(e^{i\theta})+ \Im^2 \Phi(e^{i\theta})=1$ on $[e^{- i \theta'_0}, e^{i \theta'_0}].$ Differentiating the last inequality at $\theta'=0$ gives 
$$ \Re \Phi(1) \partial_y  \Re \Phi(1)+ \Im \Phi(1) \partial_y  \Im \Phi(1)=0  $$
which combined with $\Phi(1)=1$ implies that $\partial_y \Re \Phi(1) =0.$ The last equality coupled with the fact that $\Phi$ is holomorphic implies 
$$\det D_{(x,y)} \Phi(1)= \big(\partial_x \Re \Phi(1) \big)^2+ \big(\partial_y \Re \Phi(1) \big)^2=\big(\partial_x \Re \Phi(1) \big)^2.$$
As a result, we have $\partial_x \Re \Phi(1)  \not = 0.$ On the other hand,  since 
$$|\Phi(1)|^2= 1 = \max_{x \in [0,1]} |\Phi(x)|^2,$$
we have 
$$0 \le \partial_x |\Phi(x)|^2 |_{x=1}= \Re \Phi(1) \partial_x  \Re \Phi(1) + \Im \Phi(1) \partial_x \Im \Phi(1)= \partial_x \Re \Phi(1).$$
Hence, one gets $\partial_x \Phi(1) >0.$  Direct computations gives
$$\partial_x u'_0(1)= \partial_x u(1) \partial_x \Re \Phi(1)+ \partial_y u(1) \partial_x \Im \Phi(1)=-\partial_x \Re \Phi(1)<0.$$ 
Define $u_0:= u'_0/ \partial_x \Re \Phi(1).$ We obtain $\partial_x u_0(1) = -1$ and $u_0(e^{i\theta}) =0$ for  $|\theta| \le \theta_0'.$  It remains to check that 
\begin{align} \label{ine_caicuoidayu0}
u_0(z) \ge c''(1- |z|),
\end{align}
 for some constant $c''>0.$  Since $u_0(z)>0$ and $u(z) >0$ on $\D$ and $\partial \Omega  \cap \overline{\D}= [e^{-i \theta_0/2}, e^{i \theta_0/2}],$ it is enough to check (\ref{ine_caicuoidayu0}) for $z$ so that $w=\Phi(z)$ is close to $[e^{-i \theta_0/2}, e^{i \theta_0/2}].$ Let  $w= \Phi(z) \in \Omega$  close to $[e^{-i \theta_0/2}, e^{i \theta_0/2}].$ By our choice of $\Omega,$ the axe $Ow$ is transverse to $\partial \Omega$ at a unique point $w'= \Phi(z')$ for $z' \in \partial \D.$ The $\mathcal{C}^1$- boundedness of $\Phi^{-1}$ imply that $|w- w'| \ge  c_1 |z-z'|$  for some constant $c_1$ independent of $(z,z').$ On the other hand, since $\Omega  \subset \D,$ we have $|w- w'| \le 1 - |w|.$ Hence, 
$$1 - |w| \ge c_1 |z- z'| \ge c_1(1 -|z|),$$
because $z' \in \partial \D.$ Write $w= |w| e^{i\theta_w}.$ Note that $ \theta_w \in (\theta_0, \theta_0)$ if $w$ is close enough to $[e^{-i \theta_0/2}, e^{i \theta_0/2}].$ We deduce that
$$u'_0(z)= u(\Phi(z))= u(w) \ge c'(1 - |w|) \ge c' c_1 (1 - |z|).$$   
Hence, one gets (\ref{ine_caicuoidayu0}). The proof is finished.
\end{proof}

We are now ready to introduce the Bishop equation which allows us to construct the promised family of analytic discs. Let  $u_0$ be a function described in Lemma \ref{le_version8existenceu_MA} and $\theta_{u_0}$ be the constant there.  Let $\boldsym{\tau}_1, \boldsym{\tau}_2 \in \overline{\B}_{n-1} \subset \R^{n-1}.$ Define $\boldsym{\tau}^*_1:= (1, \boldsym{\tau}_1) \in \R^n$ and $\boldsym{\tau}^*_2:= (0, \boldsym{\tau}_1) \in \R^n$ and $\boldsym{\tau}:=(\boldsym{\tau}_1,\boldsym{\tau}_2).$ Let $t$ be a positive number in $(0,1)$ which plays a role as a scaling parameter in the equation (\ref{Bishoptype}) below. 

In order to construct an analytic disc partly attached to $K$, it suffices to find a  map 
$$U: \partial \D \rightarrow \B_n \subset \R^n,$$
which is H\"older continuous, satisfying the following Bishop-type equation 
\begin{align}\label{Bishoptype}
U_{\boldsym{\tau},t}(\xi)= t\boldsym{\tau}^*_2 - \mathcal{T}_1\big(h(U_{\boldsym{\tau},t}) \big)(\xi) - t\mathcal{T}_1 u_0(\xi) \, \boldsym{\tau}^*_1,
\end{align}      
Indeed, suppose that (\ref{Bishoptype}) has a solution. For simplicity,  we use the same notation $U_{\boldsym{\tau},t}(z)$ to denote the harmonic extension of $U_{\boldsym{\tau},t}(\xi)$ to $\D.$  Let $P_{\boldsym{\tau},t}(z)$ be  the harmonic extension of $h\big(U_{\boldsym{\tau},t}(\xi)\big)$ to $\D.$ %One should not confuse $P_{\mathbf{z},t}(z)$ with $h\big(U_{\mathbf{z},t}(z)\big).$  
Define
$$F(z, \boldsym{\tau},t) := U_{\boldsym{\tau},t}(z)+ i  P_{\boldsym{\tau},t}(z)+ i  t \, u_0(z) \, \boldsym{\tau}^*_1$$
which is a family of analytic discs parametrized by $(\boldsym{\tau},t).$ For any $\xi \in [e^{-i\theta_{u_0}}, e^{i\theta_{u_0}}],$ the defining formula of $F$ and the fact that $u_0 \equiv 0$ on $ [e^{-i\theta_{u_0}}, e^{i\theta_{u_0}}]$ imply that  
$$F(\xi, \boldsym{\tau},t)=U_{\boldsym{\tau},t}(\xi)+ i  P_{\boldsym{\tau},t}(\xi)=U_{\boldsym{\tau},t}(\xi)+ i h\big(U_{\boldsym{\tau},t}(\xi)\big) \in K.$$
In other words, $F$ is $[e^{-i\theta_{u_0}}, e^{i\theta_{u_0}}]$-attached to $K$. %Moreover we have  $$F(1, \boldsym{\tau},t) =t \boldsym{\tau}^*_2  + i h(t\boldsym{\tau}^*_2).$$
In what follows, it is convenient to regard  $U_{\boldsym{\tau},t}(z)$ as a function of the variable $(z,\boldsym{\tau}).$ 

\begin{proposition}\label{pro_BishopequationMA}
There  are a positive number $t_1 \in (0,1)$ and a real number $c_1>0$  satisfying the following property.  For any $t \in (0, t_1]$ and any $\boldsym{\tau} \in \overline{\B}_{n-1}^2,$ the equation (\ref{Bishoptype}) has a unique solution $U_{\boldsym{\tau},t}$ which is  $\mathcal{C}^{2, \frac{1}{2}}$ in $(\xi, \boldsym{\tau})$ and such that
\begin{align}\label{ine_danhgiachuancuaU}
\|D^j_{(\xi,\boldsym{\tau})} U_{\boldsym{\tau},t}\|_{\mathcal{C}^{\frac{1}{2}}(\partial \D)} \le c_1  t,
\end{align} 
for any $\boldsym{\tau} \in \overline{\B}_{n-1}^2$ and $j=0,1$ or $2,$ where $D_{(\xi,\boldsym{\tau})}$ is the differential with respect to both $(\xi, \boldsym{\tau})$  and $D_{(\xi,\boldsym{\tau})}^2:= D_{(\xi,\boldsym{\tau})} \circ D_{(\xi,\boldsym{\tau})}.$  
%and\begin{align}\label{ine_danhgiachuancuaUtai1-delta}|D_{\mathbf{z}} U_{\mathbf{z},t}(1-\delta)| \le t \quad \text{for } \delta \in (0,1).\end{align}
\end{proposition}

\begin{proof}  This is a direct consequence of a general result due to Tumanov, see  \cite[Th. 4.19]{MerkerPorten2} or see \cite[Pro. 4.2]{Vu_feketepoint} for a more simple proof adapted to our present situation.
\end{proof}

Let $U_{\boldsym{\tau},t}$ be the unique solution of (\ref{Bishoptype}). As above we also use $U_{\boldsym{\tau},t}(z)$ to denote its harmonic extension to $\D.$ Let $F(z, \boldsym{\tau},t)$ and $P_{\boldsym{\tau},t}$  be as above.  Our goal is to study the behaviour of the image of the family $F(\cdot, \boldsym{\tau},t)$ near  $K,$ or in other words when $z$ is close to $[e^{-i\theta_{u_0}}, e^{i\theta_{u_0}}] \subset \partial \D.$

\begin{lemma}\label{le_danhgiadaohamcuaP'tauMA} There exists a constant $c_2$ so that for every $t \in (0, t_1]$ and every $(z,\boldsymbol{\tau}) \in \overline{\D} \times \overline{\B}_{n-1}^2,$ we have
\begin{align}\label{ine_danhgiachuancuaP'2}
\| D^j_{(z,\boldsym{\tau})} U_{\boldsymbol{\tau},t}(z)\| \le c_2 t \quad  \text{and} \quad  \| D^j_{(z,\boldsym{\tau})} P_{\boldsymbol{\tau},t}(z)\| \le c_2 t^2,
\end{align} 
for $j=0,1,2.$
\end{lemma}

\begin{proof} In view of  (\ref{ine_danhgiachuaCkcuauvoibien}) and (\ref{ine_danhgiachuancuaU}), the first inequality of (\ref{ine_danhgiachuancuaP'2}) is obvious and for the second one,  it is enough to estimate the $\mathcal{C}^{1/2}(\partial \D)$-norms of  $D^j_{(\xi,\boldsym{\tau})} P_{\boldsym{\tau},t}(\xi)$  for $j=0,1,2.$ Since  $P_{\boldsymbol{\tau},t}(\xi)= h\big(U_{\boldsymbol{\tau},t}(\xi)\big)$ on $\partial \D,$ we have 
$$\partial_{\xi}P_{\boldsymbol{\tau},t}(\xi) = Dh\big(U_{\boldsymbol{\tau},t}(\xi)\big) \partial_{\xi}U_{\boldsymbol{\tau},t}(\xi).$$  
This combined with (\ref{ine_chuanC3cuahdanhgialuythua3}) and (\ref{ine_danhgiachuancuaU}) yields that 
$$\|\partial_{\xi}P'_{\mathbf{z},t,\boldsymbol{\tau}}\|_{\mathcal{C}^{1/2}(\partial \D)} \le c_0 \|U_{\boldsymbol{\tau},t}\|_{\mathcal{C}^{1/2}(\partial \D)} \, \|\partial_{\xi}U_{\boldsymbol{\tau},t}\|_{\mathcal{C}^{1/2}(\partial \D)} \le  c_0 c_1 t^2.$$
By similar arguments, we also have $|\partial^j_{\xi}P_{\boldsymbol{\tau}}(\xi)| \lesssim  t^2$ with $j=0,2.$ To deal with the other partial derivatives, observe that for $0 \le j \le 2,$  $D^j_{\boldsymbol{\tau}} P_{\boldsymbol{\tau},t}$ is the harmonic extension of $D^j_{\boldsymbol{\tau}} h \big( U_{\boldsymbol{\tau},t}(\cdot) \big)$ to $\D.$ Hence, in order to estimate $ D^k_{z} D^j_{\boldsymbol{\tau}} P_{\boldsymbol{\tau},t}$  for $0 \le k,j \le 2,$ we can apply the same reasoning as above. Thus the proof is finished. 
\end{proof}

\begin{proposition} \label{pro_corverKtau1fixed} There are  three constants  $t_2 \in (0,t_1],$ $\theta_0  \in (0, \theta_{u_0})$ and $\epsilon_0>0$ such that  for any $\boldsym{\tau}_1 \in \overline{\B}_{n-1}$ and $t \in (0, t_2]$ the map $F(\cdot, \boldsym{\tau}_1,t): [e^{-i\theta_{0}}, e^{i\theta_{0}}] \times \overline{\B}_{n-1} \rightarrow K$ is a diffeomorphism onto its image which contains the graph of $h$ over $\B_n(0,  t \epsilon_0).$  
\end{proposition}

\begin{proof} By Cauchy-Riemann equations, we have 
$$\partial_y U_{\boldsym{\tau},t}(1)= -  t \partial_x u_0(1) \boldsym{\tau}^*_1-  \partial_x P_{\boldsym{\tau},t}(1)= t \boldsym{\tau}^*_1-  \partial_x P_{\boldsym{\tau},t}(1).$$   
The last term is $O(t^2)$ by Lemma \ref{le_danhgiadaohamcuaP'tauMA}. Thus the first component of $\partial_y U_{\boldsym{\tau},t}(1)$ is greater than $t/2$ provided that $t \le t_2$ small enough. A direct computation gives  $\partial_y U_{\boldsym{\tau},t}(1)=\partial_{\theta} U_{\boldsym{\tau},t}(1).$ Consequently, the first component of $\partial_{\theta} U_{\boldsym{\tau},t}(1)$ is greater than $t/2$ for $t \le t_2.$   

 On the other hand,  by (\ref{Bishoptype}), we have $U_{\boldsym{\tau},t}(1)=t \boldsym{\tau}_2^*$ which implies  $\partial_{\boldsym{\tau}_2}U_{\boldsym{\tau},t}(1)$ is a $(n,n-1)$ matrix whose the fist row is $0$ and the other rows form the identity matrix. Combining with the above argument shows that $D_{\boldsym{\tau}_2, \theta}U_{\boldsym{\tau},t}(1)$ is a nondegenerate matrix. This coupled with the fact that $F(e^{i\theta}, \boldsym{\tau}_1,t)= U_{\boldsym{\tau},t}(e^{i\theta})$ for $\theta \in [-\theta_0,\theta_0]$ implies the desired result. The existence of $\epsilon_0$ is obvious. The proof is finished.
\end{proof}

For $a \in \C^n$ and $A \subset \C^n,$ $\dist(a,A)$ denotes the distance from $a$ to $A.$  
  
\begin{proposition} \label{pro_corverCndung Psi} There are  two constants  $t_3 \in (0,t_2],$ $r_0>0$ such that for every $t \in (0, t_3),$  the restriction $F_1$ of  $F$ to $\big(\B_2(1,r_0) \cap \D\big) \times \overline{\B}_{n-1}^2$ is a diffeomorphism onto its image and  for any $(z, \boldsym{\tau}),$  we have 
\begin{align} \label{ine_detDF0}
\big|\det D F_1(z, \boldsym{\tau},t) \big| \gtrsim t^{2n} \big[1- |z|\big]^{n-1}
\end{align}
and
\begin{align} \label{ine_distancetoKhF}
t(1 - |z|) \lesssim \dist\big(F_1(z, \boldsym{\tau},t), K'\big) \lesssim t(1 - |z|).
\end{align}
\end{proposition}

\begin{proof} Let $r_0, t_3$ be two positive small constants to be chosen later. For the moment, we take $r_0$ to be small enough so that if $z= |z| e^{i\theta} \in \B_2(1,r_0) \cap \D,$ then $\theta \in (\theta_0, \theta_0),$ thus we have $u_0(e^{i\theta})=0.$   Fix a constant $t \in (0,t_3].$ Provided that $t_3$ and $r_0$ are small enough we will prove in the order  (\ref{ine_distancetoKhF}), (\ref{ine_detDF0}) and finally that $F_1$ is a diffeomorphism. Extend $h$ to be a $\mathcal{C}^3$ function on $\R^n.$ Let $\Psi: \C^n \rightarrow \C^n$ defined by 
$$\Psi(\mathbf{x}+ i \mathbf{y}):= \mathbf{x}+ i \mathbf{y} - i h (\mathbf{x}).$$
One can see without difficulty that $\Psi$ is a diffeomorphism sending $K'$ to $\B_n,$ where we embed 
$$\R^n   \hookrightarrow \R^n+ i \R^n= \C^n.$$ 
Let $F'_1:= \Psi \circ F_1.$ We have 
\begin{align} \label{eq_Fnga1thucao}
\Im F'_1(z, \boldsym{\tau},t)= P_{\boldsym{\tau},t}(z) - h\big(U_{\boldsym{\tau},t}(z) \big) + t u_0(z) \boldsym{\tau}_1^*\quad \text{and} \quad \Re F'_1(z, \boldsym{\tau},t)(z)= U_{\boldsym{\tau},t}(z).
\end{align}
 By the above property of $\Psi$, it suffices to prove the required property for $(F'_1, \B_n)$ in place of $(F_1,K').$ Note that $P_{\boldsym{\tau},t}(z)$ and 
$h\big(U_{\boldsym{\tau},t}(z) \big)$ are identical on $\partial \D.$ This together with (\ref{ine_danhgiachuancuaP'2}) yields
\begin{align} \label{eq_taylortbinh}
P_{\boldsym{\tau},t}(z) - h\big(U_{\boldsym{\tau},t}(z) \big)=t^2(1 - |z|)R_0(z, \boldsym{\tau},t),
\end{align}
where $R_0(z, \boldsym{\tau},t ) $ is $\mathcal{C}^1$ in $(z, \boldsym{\tau})$ so that $\|R_0(\cdot,t) \|_{\mathcal{C}^1} \lesssim 1.$ Remember that $t$ is fixed, so we do not consider it as a variable when taking the $\mathcal{C}^1$ norm. On the other hand, by our choice of $u_0$ and Lemma \ref{le_version8existenceu_MA}, one has $u_0(z) \gtrsim (1- |z|).$ By this and (\ref{eq_taylortbinh}) and (\ref{eq_Fnga1thucao}),  we obtain
\begin{align*}
\dist\big(F_1(z, \boldsym{\tau},t), K'\big) &\gtrsim \dist\big(F'_1(z, \boldsym{\tau},t), \R^n \big)\\
&= |\Im F'_1(z, \boldsym{\tau},t)| \gtrsim t(1-|z|) |\boldsym{\tau}_1^*|- t^2(1- |z|).
\end{align*}
Thus if $t$ is sufficiently small, the first inequality of (\ref{ine_distancetoKhF}) follows. 

For $t_3$ small enough, $U_{\boldsym{\tau},t}(z)  \in \B_n.$ Hence, we get
$$\dist\big(F_1(z, \boldsym{\tau},t), K'\big) \lesssim \dist\big(F'_1(z, \boldsym{\tau},t), \B_n\big) \lesssim | \Im F'_1(z, \boldsym{\tau},t)|.$$    
Write $z=|z| e^{i\theta} \in \B_2(1, r_0) \cap \D.$ Hence $\theta \in [-2 r_0, 2 r_0] \subset (\theta_0, \theta_0)$ if $r_0$ is small enough. Since $u_0(e^{i\theta})=0,$ we deduce from (\ref{eq_Taylorhuongru}) that
\begin{align} \label{eq_tayloru_0R12}
u_0(z)= (1-  |z|)+  \theta (1 - |z|) R_1(z)  + (1- |z|)^2 R_2(z),
\end{align}
where $R_j$ is smooth function with $\|R_j\|_{\mathcal{C}^1} \lesssim 1$ for $j=1,2.$ Put $\epsilon:=\max\{2 r_0, t\}.$ We choose $(t,r_0)$ to be so small that $\epsilon<1/2.$ Put 
\begin{align} \label{def_dinhgnhiaT0}
T_0(z, \boldsym{\tau},t):= tR_0(z, \boldsym{\tau},t)+ \big(\theta  R_1(z)  + (1- |z|) R_2(z) \big) \boldsym{\tau}_1^*
\end{align}
which satisfies
\begin{align} \label{ine_Tkhong}
\|T_0\|_{\mathcal{C}^0} \lesssim \epsilon, \quad \| D_{\boldsym{\tau}}T_0\|_{\mathcal{C}^0} \lesssim \epsilon
\end{align}
because $|\theta| \le 2 r_0$ and $1- |z| \le r_0.$ Combining (\ref{eq_tayloru_0R12}), (\ref{eq_taylortbinh}) and (\ref{eq_Fnga1thucao}) gives  
\begin{align} \label{eq2_rhoFnga}
\Im F'_1(z, \boldsym{\tau},t)=  t (1- |z|)\big[\boldsym{\tau}_1^*+ T_0(z, \boldsym{\tau},t)  \big].
\end{align}     
Consequently, using (\ref{ine_Tkhong}) we obtain
$$|\Im F'_1(z, \boldsym{\tau},t)| \lesssim t (1- |z|)$$
which proves  the second inequality of (\ref{ine_distancetoKhF}).

 By (\ref{Bishoptype}) and the Cauchy-Riemann equations, we have $U_{\boldsym{\tau},t}(1)= t \boldsym{\tau}_2^*$ and
 $$\partial_{y} U_{\boldsym{\tau},t}(z)=- \partial_x P_{\boldsym{\tau},t}(z)- t \partial_x u_0(z) \boldsym{\tau}_1^*$$
 and 
 $$\partial_{x} U_{\boldsym{\tau},t}(z)= \partial_y P_{\boldsym{\tau},t}(z)+ t \partial_y u_0(z) \boldsym{\tau}_1^*.$$
Observe that 
$$\partial_{\theta} U_{\boldsym{\tau},t}(e^{i\theta})=- \partial_{x} U_{\boldsym{\tau},t}(e^{i\theta})\sin \theta +\partial_{y} U_{\boldsym{\tau},t}(e^{i\theta}) \cos \theta .$$
These above equalities combined with  (\ref{ine_danhgiachuancuaP'2})   and  Taylor's expansion to $U_{\boldsym{\tau},t}(e^{i\theta})$ at $\theta=0$ gives 
\begin{align} \label{eq_Utautkhaitrientai1}
U_{\boldsym{\tau},t}(e^{i\theta})= t \boldsym{\tau}_2^* + t^2 R_3(\theta,\boldsym{\tau},t) +t \theta \boldsym{\tau}_1^*+ t \theta^2 R_4(\theta)\boldsym{\tau}_1^*,
\end{align}
where 
$$R_3(\theta,\boldsym{\tau},t):=\int_{0}^{\theta}\big[\partial_y P_{\boldsym{\tau},t}(e^{i s }) \cos s -\partial_x P_{\boldsym{\tau},t}(e^{i s}) \sin s \big] ds$$
which is of $\mathcal{C}^1$ norm $\lesssim 1,$ and
$R_4(\theta)$ is a $\mathcal{C}^1$ function satisfying $\|R_4\|_{\mathcal{C}^1} \lesssim 1.$ Remark that in (\ref{eq_Utautkhaitrientai1}), we used the $\mathcal{C}^3$ differentiability of $u_0$ and $R_4$ comes from the remainder of the Taylor expansion of $u_0$ at $1$ up to the order $2.$   

Using (\ref{eq_Utautkhaitrientai1}), Taylor's expansion for $\Re F'_1(z, \boldsym{\tau},t)$ at $\tilde{z}= e^{i\theta}$ implies 
\begin{align} \label{eq3_rhoFnga}
\Re F'_1(z, \boldsym{\tau},t)= t \boldsym{\tau}_2^* +t \theta \boldsym{\tau}^*_1+  t^2   R_3(\theta,\boldsym{\tau},t)+t \theta^2 R_4(\theta)\boldsym{\tau}_1^*+t (1 - |z|) R_5 (z, \boldsym{\tau},t),
\end{align}     
for some $\mathcal{C}^1$ function $R_5(z, \boldsym{\tau},t ) $  with $\|R_j(\cdot,t) \|_{\mathcal{C}^1} \lesssim 1.$ Define
\begin{align} \label{def_dinhgnhiaTmot}
T_1(z, \boldsym{\tau},t):= t   R_3(\theta,\boldsym{\tau},t)+\theta^2 R_4(\theta)\boldsym{\tau}_1^*+ (1 - |z|) R_5 (z, \boldsym{\tau},t),
\end{align}
which satisfies
\begin{align} \label{ine_Tmot}
\|D_{\boldsym{\tau},\theta}T_1\|_{\mathcal{C}^0} \lesssim \epsilon,
\end{align}
where we use the polar coordinate $(|z|, \theta)$ for $z.$ Combining (\ref{eq3_rhoFnga}), (\ref{ine_Tmot}), (\ref{eq2_rhoFnga}) and (\ref{ine_Tkhong}) gives (\ref{ine_detDF0}).

%where $\mathbf{a}:=(z, \boldsym{\tau},t)$ and $\mathbf{a}':=(z', \boldsym{\tau}',t).$ 

Let $\boldsym{\rho}= t \boldsym{\rho}_2+ i \, t \boldsym{\rho}_1$ be an arbitrary point in the image of $F'.$ This means that 
\begin{align} \label{eq_rhoz0tau0}
\boldsym{\rho}= F'_1(z^0, \boldsym{\tau}^0,t),
\end{align}
 for some $(z^0, \boldsym{\tau}^0).$ Let $\theta^0 \in (-\pi/2, \pi/2)$ be the argument of $z^0.$ Then $z^0= |z^0| e^{i\theta^0}.$ We will prove that the equation 
\begin{align} \label{eq_rhoFnga}
F'_1(z, \boldsym{\tau},t)=\boldsym{\rho}
\end{align}
has a unique solution, \emph{i.e} $F'_1$ is injective.   The equation (\ref{eq_rhoFnga}) is equivalent to the system of the two following equations 
\begin{align} \label{eq_loquaFrho}
\Re F'_1(z, \boldsym{\tau},t)= t \boldsym{\rho}_2
\end{align}
 and 
 \begin{align} \label{eq_loquaFrho2}
\Im F'_1(z, \boldsym{\tau},t)= t \boldsym{\rho}_1.
\end{align}
Write $T_j=(T_{j1}, \cdots, T_{jn})$ for $j=0$ or $1$ and $\boldsym{\rho}_j=(\rho_{j1}, \cdots, \rho_{jn})$ for $j=1,2.$   Define  
$$\tilde{\boldsym{\rho}}_1:= \frac{\boldsym{\rho}_1}{1- |z|} \cdot$$
 We also write $\tilde{\boldsym{\rho}}_1= (\tilde{\rho}_{11}, \cdots, \tilde{\rho}_{1n}).$ Recall that $\boldsym{\tau}^*_j=(1, \boldsym{\tau}_j)$ for $j=1$ or $2$ and $\boldsym{\tau}_j=(\tau_{j1}, \cdots, \tau_{j(n-1)}).$  We have 
\begin{align} \label{eq_zphay}
\tilde{\rho}_{1k} -  \tilde{\rho}_{11} \, \frac{\rho_{1k}}{\rho_{11}}=0, 
\end{align}
for $2\le k \le n.$ The variable $(\tilde{\rho}_{11},\theta)$ will be used as a substitute for $z.$   If $(z, \boldsym{\tau},t)$ is a solution of (\ref{eq_rhoFnga}),  identifying the first component of (\ref{eq2_rhoFnga}) and (\ref{eq_loquaFrho2}) imply
$$1+ T_{01}(z, \boldsym{\tau},t)= \tilde{\rho}_{11}$$
which in turn yields $|\tilde{\rho}_{11} -1| \lesssim \epsilon$ by (\ref{ine_Tkhong}).  Hence if $(z, \boldsym{\tau},t)$ is a solution of (\ref{eq_rhoFnga}), we get 
\begin{align} \label{ine_zphay1}
1/2 \le \tilde{\rho}_{11} \le 3/2.
\end{align}
 By (\ref{eq_loquaFrho2}) again and the fact that $\boldsym{\tau}_1 \in \B_{n-1}$,  one also gets 
\begin{align} \label{ine_zphay2}
\big| \frac{\rho_{1k}}{\rho_{11}} \big|\approx |\tau_{1(k-1)}|  \le 3/2,
\end{align}
for $2 \le k \le n.$
%since $1- |z|=\rho'_{11}/y'_{1},$ 0 < \rho'_{11} \le 2 r_0.
Since $z= |z| e^{i \theta},$ we have  
$$z=\big(1-\frac{\rho_{11}}{\tilde{\rho}_{11}} \big) e^{i\theta}.$$
From now on, we will consider $T_0,T_1$ as functions of $(\tilde{\rho}_{11}, \theta, \boldsym{\tau}).$  Define $$G=(G_1, G_2,G_3):  \overline{\B}_{n-1}^2 \times [\frac{1}{2}, \frac{3}{2}] \times \R^{n-1} \times [-2r_0, 2r_0] \rightarrow \R^{n} \times \R^n \times \R^{n-1}$$
by putting
$$G_1(\boldsym{\tau},\tilde{\boldsym{{\rho}}}_{1}, \theta):= \boldsym{\tau}_1^* +   T_0(\theta, \tilde{\rho}_{11}, \boldsym{\tau},t)- \tilde{\boldsym{{\rho}}}_{1}, \quad G_2(\boldsym{\tau},\tilde{\boldsym{{\rho}}}_{1}, \theta):=\boldsym{\tau}_2^* + \theta \boldsym{\tau}^*_1+ T_1(\theta, \tilde{\rho}_{11}, \boldsym{\tau},t) - \boldsym{\rho}_2 $$
and  
$$G_3(\boldsym{\tau},\tilde{\boldsym{{\rho}}}_{1}, \theta):=\big(\tilde{\rho}_{12} -  \tilde{\rho}_{11} \, \frac{\rho_{1k}}{\rho_{11}}, \cdots,\tilde{\rho}_{1n} -  \tilde{\rho}_{11} \, \frac{\rho_{1k}}{\rho_{11}}\big).$$ 
By (\ref{eq_zphay}), (\ref{eq3_rhoFnga}) and (\ref{eq2_rhoFnga}), resolving the system (\ref{eq_loquaFrho})-(\ref{eq_loquaFrho2}) is equivalent to finding 
$(\boldsym{\tau}, \tilde{\boldsym{{\rho}}}_{1},\theta) $ for which 
\begin{align} \label{eq_equivalentGtau}
G(\boldsym{\tau},\tilde{\boldsym{{\rho}}}_{1}, \theta)=0.
\end{align} 
By (\ref{eq_rhoz0tau0}), we know that $\mathbf{a}^0:= (\boldsym{\tau}^0, \tilde{\boldsym{{\rho}}}_{1}^0, \theta^0)$ is a solution of (\ref{eq_equivalentGtau}), where 
$$\tilde{\boldsym{{\rho}}}^0_{1}:= \frac{\boldsym{\rho}_1}{1- |z^0|}.$$
Suppose that there is an another solution $\mathbf{a}= (\boldsym{\tau}, \tilde{\boldsym{\rho}}_1, \theta)$ of (\ref{eq_equivalentGtau}).  By a direct computation, one gets
$$\partial_{\tilde{\rho}_{11}}(1 - |z|)= - \frac{\rho_{11}}{\tilde{\rho}_{11}^2}= -(1- |z|)\tilde{\rho}_{11}^{-1}= O(1- |z|) \lesssim \epsilon$$
by  (\ref{ine_zphay1}). This coupled with (\ref{def_dinhgnhiaT0}) and (\ref{def_dinhgnhiaTmot}) yields
\begin{align} \label{ine_TkhongchuanC1}
|T_0(\mathbf{a},t) - T_0(\mathbf{a}^0,t)|  \lesssim \epsilon |\mathbf{a}- \mathbf{a}^0|+ |\theta - \theta'|.
\end{align}
and 
\begin{align} \label{ine_TkhongchuanC1hai}
|T_1(\mathbf{a},t) - T_1(\mathbf{a}^0,t)|  \lesssim \epsilon |\mathbf{a}- \mathbf{a}^0|.
\end{align}
Using (\ref{ine_TkhongchuanC1hai}) and identifying the first component of the equation $G_2(\boldsym{\tau},\tilde{\boldsym{\rho}}_1, \theta)=0$ imply
\begin{align} \label{ine_y1phayhieu}
|\theta - \theta^0| \le |T_1(\mathbf{a},t) - T_1(\mathbf{a}^0,t)| \lesssim \epsilon |\mathbf{a}- \mathbf{a}^0|.
\end{align}
By doing the same thing for $G_1(\boldsym{\tau},\tilde{\boldsym{\rho}}_1, \theta)=0$ and using (\ref{ine_y1phayhieu}), we also obtain
\begin{align*} %\label{ine_y1phayhieu}
|\tilde{\rho}_{11} - \tilde{\rho}_{11}^0| \le |T_0(\mathbf{a},t) - T_0(\mathbf{a}^0,t)|  \lesssim \epsilon |\mathbf{a}- \mathbf{a}^0|.
\end{align*}
 Using the last inequality, the equality $G_3(\boldsym{\tau},\tilde{\boldsym{\rho}}_1, \theta)=0$ and (\ref{ine_zphay2}), one infers
\begin{align} \label{ine_y1phayhieu2}
| \tilde{\boldsym{\rho}}_1 -  \tilde{\boldsym{\rho}}_1^0| \lesssim  |\tilde{\rho}_{11} - \tilde{\rho}_{11}^0| \lesssim \epsilon |\mathbf{a}- \mathbf{a}^0|.
\end{align}
Similarly, using $G_1(\boldsym{\tau},\tilde{\boldsym{\rho}}_1, \theta)=0$ gives
\begin{align} \label{ine_y1phayhieu3}
|\boldsym{\tau}_1 - \boldsym{\tau}^0_1| \le |T_0(\mathbf{a},t) - T_0(\mathbf{a}^0,t)| + | \tilde{\boldsym{\rho}}_1 -  \tilde{\boldsym{\rho}}_1^0| \lesssim \epsilon |\mathbf{a}- \mathbf{a}^0|. 
 \end{align} 
Finally, using $G_2(\boldsym{\tau},\tilde{\boldsym{\rho}}_1, \theta)=0$ gives
\begin{align} \label{ine_y1phayhieu4}
|\boldsym{\tau}_2 - \boldsym{\tau}^0_2| \lesssim \epsilon |\mathbf{a}- \mathbf{a}^0|.
  \end{align} 
Summing the inequalities from (\ref{ine_y1phayhieu}) to (\ref{ine_y1phayhieu4}) and taking into account that 
$$|\mathbf{a}- \mathbf{a}^0| \le |\boldsym{\tau}_2 - \boldsym{\tau}^0_2| + |\boldsym{\tau}_1 - \boldsym{\tau}^0_1|+ | \tilde{\boldsym{\rho}}_1 -  \tilde{\boldsym{\rho}}_1^0| + |\theta- \theta^0|$$
show that $\mathbf{a}= \mathbf{a}^0.$ This means that (\ref{eq_equivalentGtau}) has a unique solution, or equivalently, so does  (\ref{eq_rhoFnga}) if $r_0$ and $t$ are small enough. %The inequality (\ref{ine_detDF0}) is clear in the view of (\ref{eq2_rhoFnga}) and (\ref{eq3_rhoFnga}). 
The proof is finished.
%\begin{align} \label{systemtau}
%\begin{cases} 
%\boldsym{\tau}_1^* +   R'_0(\theta, \tilde{\boldsym{\rho}}_1, \boldsym{\tau},t) &= \tilde{\boldsym{\rho}}_1 \\ 
%\boldsym{\tau}_2^* + \theta \boldsym{\tau}^*_1+ R'_1(\theta, \tilde{\boldsym{\rho}}_1, \boldsym{\tau},t) &=  \boldsym{\rho}'_2\\
%y'_k -  y_1' \, \boldsym{\rho}'_{1k}/ \boldsym{\rho}'_{11} &=0, \, \forall 2 \le k \le n. \end{cases}\end{align}
%Consider (\ref{systemtau}) as a system of the variable $\boldsym{\tau}=(\boldsym{\tau}_1, \boldsym{\tau}_2)$  and $(\mathbf{z}', \theta)$ are the parameter for the moment. Suppose that $\boldsym{\tau}=(\boldsym{\tau}_1, \boldsym{\tau}_2)$ and $\tilde{\boldsym{\tau}}=(\tilde{\boldsym{\tau}}_1, \tilde{\boldsym{\tau}}_2)$ are two solutions of  (\ref{systemtau}) in $\B_{n-1}(0,1)^2.$  In view of (\ref{systemtau}), one obtains 
%$$ |\tilde{\boldsym{\tau}} -\boldsym{\tau}| \le  \big| D_{\boldsym{\tau}} R'_0(\theta, \mathbf{z}', \boldsym{\tau},t) \big|  |\tilde{\boldsym{\tau}} -\boldsym{\tau}| +\big| D_{\boldsym{\tau}} R'_0(\theta, \mathbf{z}', \boldsym{\tau},t) \big|  |\tilde{\boldsym{\tau}} -\boldsym{\tau}|.$$
%By (\ref{ine_DRphayj}), the last inequality holds only if $\tilde{\boldsym{\tau}} =\boldsym{\tau}$ when $r_0$ and $t_3$ are small enough.  Thus (\ref{systemtau}) has at most a solution. By which is denoted by $\boldsym{\tau}_1(\mathbf{z}', \theta)$
\end{proof}

%By choosing $\theta_0$ in Proposition \ref{pro_corverKtau1fixed} small enough, we can suppose that the arc $[e^{-i\theta_0}, e^{i\theta_0}]$ is contained in $\B_2(1,r_0).$  
Let $\Omega$ be a simply connected subdomain of $\D$ with smooth boundary such that $\Omega$ is strictly convex and $\overline{\Omega} \cap \overline{\D}= [e^{-i \theta_1}, e^{i\theta_1}]$ for some $\theta_1 \in (0,\theta_0)$ and $\overline{\Omega} \subset \B_2(1,r_0).$ %and for any $z \in \Omega$ close to $[e^{-i \theta_0/2}, e^{i\theta_0/2}],$   the line $Oz$ passing through $z$ and the origin $O$ is transverse to $\partial \Omega.$
By Painvel\'e's theorem as in the proof of Lemma \ref{le_version8existenceu_MA}, there is a smooth diffeomorphism $\Phi$ from $\overline{\D}$ to $\overline{\Omega}$ which is a biholomorphism from $\D$ to $\Omega$ and $\Phi(1)=1$. 

Define $\tilde{F}(z, \boldsym{\tau},t):= F\big(\Phi(z), \boldsym{\tau},t\big)$ which is again a $\mathcal{C}^{2,1/2}$ family of analytic discs partly attached to $K.$ %Using the same argument as the proof of (\ref{ine_caicuoidayu0}), 

\begin{proposition} \label{pro_corverCndung Psi2} $(i)$ There are positive constants $\tilde{\theta}_0$ and $\tilde{\epsilon}_0$ so that for every $\boldsym{\tau}_1 \in \overline{\B}_{n-1}$ and $t \in (0,t_3],$ the restriction map $\tilde{F}(\cdot, \boldsym{\tau}_1,t): [e^{-i\tilde{\theta}_0}, e^{i \tilde{\theta}_0}] \times \overline{\B}_{n-1} \rightarrow K'$ is a diffeomorphism onto its image which contains the graph of $h$ over $\B_n(0, t \tilde{\epsilon}_0).$  

$(ii)$ Let $t_3$ be the constant in Proposition \ref{pro_corverCndung Psi}. Then for any $t \in (0, t_3],$ the map $\tilde{F}(\cdot,t)$ is a diffeomorphism from $\D \times \overline{\B}_{n-1}^2$ onto its image in $ \D^{n} \subset \C^n,$ and  for any $(z, \boldsym{\tau})$ we have
\begin{align} \label{ine_detDF}
\big|\det D \tilde{F}(z, \boldsym{\tau},t) \big| \gtrsim t^{n+1} \dist^{n-1}\big(\tilde{F}(z, \boldsym{\tau},t), K'\big)
\end{align}
and 
\begin{align} \label{ine_distancetoKhF*}
t(1 - |z|) \lesssim \dist\big(\tilde{F}(z, \boldsym{\tau},t), K'\big).
\end{align}
%and 
%\begin{align} \label{ine_distancetoKhF*2}t(1 - |z|) \gtrsim \dist\big(\tilde{F}(z, \boldsym{\tau},t), K'\big) \end{align}if additionally $\arg z \in [-\theta^*_0, \theta^*_0].$   
\end{proposition}

\begin{proof} Property $(i)$ is a direct consequence of Propositions \ref{pro_corverKtau1fixed}.  By the differentiability of $\Phi^{-1}$ on $\overline{\Omega},$  we have $(1- |\Phi(z)|) \gtrsim 1- |z|$ for every $z \in \D.$ Hence, by (\ref{ine_distancetoKhF}), we get (\ref{ine_distancetoKhF*}). The inequality (\ref{ine_detDF}) follows immediately from the fact that 
$$\big|\det D F_1 (z, \boldsym{\tau},t) \big| \gtrsim t^{n+1} \dist^{n-1}\big(F_1(z, \boldsym{\tau},t), K'\big)$$
which is in turn implied by (\ref{ine_distancetoKhF}) and (\ref{ine_detDF0}). The proof is finished.
\end{proof}

Using the local coordinates of $K$ at the beginning of this section and choosing $t=t_3$, the last proposition can be rephrased as follows.

\begin{proposition} \label{pro_corverCndung Psi2global} Let $a$ be an arbitrary point of $K.$ Then there exist positive constants $\tilde{\epsilon}_a, \tilde{\theta}_a$ and a $\mathcal{C}^{2,1/2}$  diffeomorphism $\tilde{F}_a: \D \times \overline{\B}_{n-1}^2 \rightarrow X $ onto its image such that the two following properties hold: 

$(i)$ for every $\boldsym{\tau}_1 \in \overline{\B}_{n-1},$ the restriction map $\tilde{F}_a(\cdot, \boldsym{\tau}_1): [e^{-i\tilde{\theta}_a}, e^{i \tilde{\theta}_a}] \times \B_{n-1} \rightarrow K$ is a diffeomorphism onto its image which contains the graph of $h$ over $\B_K(a, \tilde{\epsilon}_a).$

$(ii)$ there is an open relatively compact neighborhood $K'_a$ of $a$ in $K$ such that  for any $(z, \boldsym{\tau}),$ we have
\begin{align} \label{ine_detDFp}
\big|\det  D \tilde{F}_a(z, \boldsym{\tau}) \big| \gtrsim \dist^{n-1}\big(\tilde{F}_a(z, \boldsym{\tau}), K'_a\big)
\end{align}
and
\begin{align} \label{ine_distancetoKhF*p}
(1 - |z|) \lesssim \dist\big(\tilde{F}_a(z, \boldsym{\tau}), K'_a\big).
\end{align}
%and\begin{align} \label{ine_distancetoKhF*p2}t(1 - |z|) \gtrsim \dist\big(F_p^*(z, \boldsym{\tau},t), K\big)\end{align}if additionally $\arg z \in [-\theta^*_p, \theta^*_p].$
\end{proposition}

%We will not use Proposition \ref{pro_corverCndung Psi2global} until the end of Subsection.

\section{Some estimates for p.s.h. functions} \label{sec_intergrability}

In this section, we will prove some key estimates for  p.s.h functions and their $dd^c$ on $\C^n.$  For a Borel subset $A$ of $\R^m$ with $m \in \N,$  denote by $|A|$ the volume of $A$ with respect to the canonical volume form $vol_{\R^m}.$  In what follows, for simplicity, we will write $\int_A f$ instead of $\int_A f d vol_{\R^m}$ for every  Borel set $A \subset \R^m$ and every integrable function $f$ on $A.$ In particular, this convention is applied to $\C^n= \R^{2n}.$     %over tubular neighborhoods of size $\epsilon$ of a generic submanifolds in $\C^n.$ %These results are of independent interest since they describe, to some extent, the behavior of p.s.h. functions near generic submanifolds.  

\begin{lemma} \label{le_volumepshonRn0tongquat} Let $V$ be an open subset of $\C^n$ and let  $V_1$ be a compact subset of $V.$ Let $\varphi$ be a p.s.h. function on $V.$ Then there exists a constant $c$ independent of $\varphi$ such that for any Borel set $V_2 \subset V_1,$ we have
\begin{align} \label{ine_volumeepsilonu0}
\int_{V_2} |\varphi| \le c  |V_2| \max\{1, - \log |V_2|  \} \int_{V} |\varphi|.
\end{align}
\end{lemma}

\begin{proof}  If $\varphi \equiv 0$ or $ \int_{V} |\varphi|= \infty,$  then there is nothing to prove. Now suppose $\varphi \not = 0$ and $ \int_{V} |\varphi| < \infty.$ Let $\varphi_1= \varphi/  \int_{V} |\varphi|.$ We have $ \int_{V} |\varphi_1|= 1.$ As a result, there exist  two positive constants $(c_1, \alpha_1)$ independent of $\varphi_1$ for which 
\begin{align} \label{ine_exepophi_1chuanL11}
\int_{V_1} e^{\alpha_1 |\varphi_1|}\le c_1.
\end{align} 
Let $\boldsym{1}_{V_2}$ be the characteristic function of $V_2.$ Let $\mu:=  |V_2|^{-1} \boldsym{1}_{V_2}  vol_{\C^n}$ which is a probability measure supported in $V_2.$ 
We have
$$\int_{V_2} |\varphi_1| = \alpha^{-1}_1 \int_{V_2} \log  e^{\alpha_1 |\varphi_1|}= \alpha^{-1}_1 |V_2|\int_{V_2} \log  e^{\alpha_1 |\varphi_1|} \, d \mu.$$
This together with the concavity of the $\log$ function implies
$$\int_{V_2} |\varphi_1|  \le  \alpha^{-1}_1 |V_2|  \log \int_{V_2}  e^{\alpha_1 |\varphi_1|} \, d \mu$$
which, by (\ref{ine_exepophi_1chuanL11}), is less than or equal to 
$$\alpha^{-1}_1  |V_2|\big( \log c_1 -\log |V_2|\big).$$
Hence (\ref{ine_volumeepsilonu0}) follows. The proof is finished.
\end{proof}

Now let $h,K'$ be as in Section \ref{sec_analyticdisc}. Let $\epsilon$ be a real positive number and  $K'_{\epsilon}$ the compact subset of $\C^n$ consisting of points of distance $\le \epsilon$ to $K'.$ Obviously, the volume of $K'_{\epsilon}$ is $\lesssim \epsilon^n.$ Using Lemma \ref{le_volumepshonRn0tongquat} for $V_2 = K'_{\epsilon},$ we get the following.

\begin{corollary} \label{cor_volumepshonRn} Let $V$ be an open subset of $\C^n$ containing $H_1.$ Let $\varphi$ be a p.s.h. function on $V.$   Then there is a constant $c$ independent of $\varphi$ for which
\begin{align} \label{ine_volumeepsilonu}
\int_{K'_{\epsilon}} |\varphi|  \le  c \epsilon^{n} |\log \epsilon | \int_{V} | \varphi| 
\end{align}
for every $\epsilon \le 1/2.$
\end{corollary}

Now we will give a similar estimate for the mass of $dd^c \varphi$ on $K'_{\epsilon}.$ We begin with a general result.

\begin{lemma} \label{le_ddcmuntru1moipp} Let $V,V_1,V_2$ be open subsets of $\C^n$  such that $V_2 \Subset V_1 \Subset V.$ Let $T$ be a closed positive current of bidimension $(p,p)$ on $V$ and $\lambda$  a real number  $>1.$ Let $\varphi$ and $\rho$ be two bounded p.s.h. functions on $V.$ %Let $\rho$ be a bounded p.s.h. function on $V.$ %Let $\delta$ be a positive constant in $(0,1)$. Assume that for every $\epsilon>0,$ we have \begin{align} \label{ine_drohdcrhohy}d \rho \wedge d^c \rho \le \epsilon^{1+\delta} \omega\end{align}on $V_1 \cap \{\rho \le \epsilon\},$ where $\omega$ is the canonical K\"ahler form on $\C^n.$  %For small  $\epsilon>0,$ denote $A_{\rho,\epsilon}$ the set of $ \mathbf{z} \in A$ for which $\rho(\mathbf{z}) \le \epsilon.$ 
Let $A$ be a subset of $V_2$ and $a_{\epsilon,\varphi}$ be an upper bound for $|\varphi|$ on $V_1 \cap \{ \rho \le \epsilon\}$ for $\epsilon>0.$ Assume that $\rho$ is bounded by $1$ on $V.$ Then there is a constant $c$ independent of $T,A,\rho, \varphi$ such that 
 \begin{align} \label{ine_ddcmuntru1bonlonpp}  
\int_{A \cap  \{\rho \le \epsilon\}} T \wedge (dd^c \varphi)^{p} \le c [\epsilon^{-1}a_{\lambda \epsilon,\varphi}]^{p} \|T\|_{V_1},
\end{align}
for every  $\epsilon \in (0,1).$
\end{lemma}

\begin{proof} We prove (\ref{ine_ddcmuntru1bonlonpp}) by induction on $p$. When $p=0,$ the conclusion is obvious. Suppose (\ref{ine_ddcmuntru1bonlonpp}) holds for  $p-1.$ We need to prove its validity for $p.$ Let $\chi$ be a smooth function compactly supported in some $V'_1 \Subset V_1$ such that $0 \le  \chi \le 1$ and $\chi \equiv 1$ on $V_2.$ Let  $\epsilon$ be a positive constant. Choose a constant $\lambda' \in (1, \lambda).$ %Observe that $\max\{x, 0\}$ is an increasing convex function on $\R$ which is smooth everywhere except at $x=0.$ Hence, we can find a smooth increasing convex function $\psi$ on $\R$ such that the two following properties hold:$(i)$  we have $\psi \ge 0$ on $\R$ and $|\psi- \max\{\cdot, 0 \}| \le \epsilon$ on $\R$ and $\psi(x)=\max\{x,0\}$ for $|x| \ge \epsilon,$   $(ii)$ we have $0 \le \psi' \le 1$ and there is a constant $c$ for which $|\psi''(x)| \le c \epsilon^{-1}$ for $|x| \le  \epsilon.$ \\
Define 
$$\rho_{\epsilon}:= \max\{\rho, \lambda'\epsilon\} - \max\{\rho,\epsilon\}$$
which is the difference of two bounded p.s.h. functions on $V$. Clearly, we have $0 \le \rho_{\epsilon} \le (\lambda'-1)\epsilon$ and $\rho_{\epsilon}= (\lambda'-1)\epsilon$ on $\{ \rho \le \epsilon\}$ and  $\rho_{\epsilon}=0$  on $\{\rho \ge  \lambda'\epsilon\}.$ This yields 
\begin{align} \label{ine_uocluongvarphivepspp}
\int_{A \cap  \{\rho \le \epsilon\}} T  \wedge (dd^c \varphi)^{p} \le \frac{\epsilon^{-1}}{\lambda'-1}  \int_{V} \chi \rho_{\epsilon} T \wedge (dd^c \varphi)^{p} \lesssim \epsilon^{-1}  \int_{V} \chi \rho_{\epsilon} T \wedge (dd^c \varphi)^{p}
\end{align}
which is, by integration by parts, equal to
\begin{align}\label{eq_sumR1R2R}
\epsilon^{-1} \int_{V} \rho_{\epsilon} \varphi dd^c \chi \wedge T \wedge (dd^c \varphi)^{p-1}+ \epsilon^{-1}\int_{V} \varphi \chi dd^c\rho_{\epsilon}\wedge T \wedge (dd^c \varphi)^{p-1}+ R,
\end{align}
where 
$$R= 2 \epsilon^{-1} \int_V \varphi d \chi \wedge d^c \rho_{\epsilon} \wedge T \wedge (dd^c \varphi)^{p-1}.$$
Denote by $R_1$ and $R_2$ respectively the first and second terms in (\ref{eq_sumR1R2R}). We are now going to estimate $R_1, R$ and finally $R_2.$ Let $\omega$ be the canonical K\"ahler form on $\C^n.$ Since $dd^c \chi \lesssim \omega$ and $|\varphi| \le a_{2\epsilon, \varphi}$ on $\supp \rho_{\epsilon},$ we get 
$$R_1 \le  \epsilon^{-1} a_{\lambda'\epsilon, \varphi} \int_{V_1' \cap \{\rho \le \lambda' \epsilon\}}  \omega \wedge T  \wedge (dd^c \varphi)^{p-1}.$$
Applying the induction hypothesis to $\omega \wedge T, \lambda' \epsilon$ in place of $T,\epsilon$ implies 
\begin{align} \label{ine_danhgiaR1T}
R_1 \le   \epsilon^{-1} a_{\lambda'\epsilon, \varphi} \int_{V'_1  \cap \{\rho \le \lambda' \epsilon\}} \omega \wedge T \wedge (dd^c \varphi)^{p-1} \lesssim [\epsilon^{-1} a_{\lambda \epsilon, \varphi}]^{p}.
\end{align}
As to $R,$ the Cauchy-Schwarz inequality applied to a suitable scalar product gives
\begin{align} \label{ine_danhgiaRrho}
|R|^2  &\le \epsilon^{-2} \int_{V_1'}  |\varphi | \boldsym{1}_{\{\rho  \le \lambda' \epsilon\}} d \chi \wedge d^c \chi \wedge T \wedge (dd^c \varphi)^{p-1} \int_{V_1'} |\varphi |d \rho_{\epsilon} \wedge d^c \rho_{\epsilon} \wedge T \wedge (dd^c \varphi)^{p-1}\\
\nonumber
& \lesssim  [\epsilon^{-1} a_{\lambda'\epsilon, \varphi}]^{p+1} \int_{V_1'} d \rho_{\epsilon} \wedge d^c \rho_{\epsilon} \wedge T \wedge (dd^c \varphi)^{p-1}
\end{align}
by induction hypothesis and the fact that $d \rho_{\epsilon} \wedge d^c \rho_{\epsilon}$ is positive and supported on $\{\rho \le \lambda' \epsilon\}.$ Denote by $R'$ the last integral. Since $\rho_{\epsilon}$ is the difference of two bounded p.s.h. functions on $V,$ so  is $\rho_{\epsilon}^2.$ More precisely, since $|\rho| \le 1$ on $V$ we can find four p.s.h function $\psi_j$ with $1 \le j \le 4$ so that they are bounded independent of $\epsilon$  and 
\begin{align} \label{ine_ddcrhorehobinh}
\rho_{\epsilon}^2=\psi_1- \psi_2 \quad \text{and} \quad \rho_{\epsilon}= \psi_3 - \psi_4.
\end{align}
We also have 
$$dd^c \rho_{\epsilon}^2= 2 d \rho_{\epsilon} \wedge d^c \rho_{\epsilon}+ 2 \rho_{\epsilon} dd^c \rho_{\epsilon}.$$
Note that each side of the last equality is well-defined. Substituting this to the defining formula of $R',$ then using (\ref{ine_ddcrhorehobinh}), one gets 
$$R' \le  \sum_{j=1}^4\int_{V_1' \cap \{\rho \le \lambda' \epsilon\}} dd^c \psi_j \wedge T \wedge (dd^c \varphi)^{p-1}$$%+ \int_{V_1' \cap \{\rho \le 2 \epsilon\}} dd^c \psi_2 \wedge T \wedge (dd^c \varphi)^{p-1}$$
which, by induction hypothesis, is $\lesssim$
$$   [\epsilon^{-1} a_{\lambda\epsilon, \varphi}]^{p-1} \sum_{j=1}^4 \| dd^c \psi_j \wedge T   \|_{V_1''},$$
where $V_1''$ be a relatively compact subset of $V_1$ which is open and  contains $\overline{V}'_1.$ By the Chern-Levine-Nirenberg inequality, the last sum is $\lesssim \|T\|_{V_1}.$ Combining with (\ref{ine_danhgiaRrho}), we obtain
\begin{align} \label{ine_danhgiaRT}
R \le  [\epsilon^{-1} a_{ \lambda \epsilon, \varphi}]^{p} \|T\|_{V_1}.
\end{align}
Bounding $R_2$ is done similarly. The proof is finished. 
\end{proof}

\begin{lemma} \label{le_tinhcuthe} Let $f$ be a real $\mathcal{C}^2$ function on an open set $V \subset \C^n.$ Let $g(t):= |t| \log (|t|+2)$ for $t\in \R.$ Let $\omega$ be the canonical K\"ahler form on $\C^n.$ Then we have
$$12 \, dd^c ( g \circ f) \ge   df \wedge d^c f -  2 n\|D^2 f\|_{L^{\infty}(V)}\, \omega$$
as currents on $V.$
\end{lemma}

\begin{proof}  By direct computations, one obtains for $t >0,$ 
$$g'(t)= 1- \frac{2}{2+t}+ \log (2+ t), \quad g''(t)=\frac{2}{(2+t)^2}+ \frac{1}{2+t}$$
and for $t<0,$
$$g'(t)= -1+ \frac{2}{2-t}- \log (2- t), \quad g''(t)=\frac{2}{(2-t)^2}+ \frac{1}{2-t}\cdot $$
For $k\ge 3,$ we are going to construct a sequence of $\mathcal{C}^2$ convex function $g_k$ of uniformly bounded $L^{\infty}$ norm converging pointwise to $g.$ To this end, we define
$$q_k(t):= \frac{2}{(2+|t|)^2}+ \frac{1}{2+|t|} \quad  \text{for $t \ge 1/k$}$$
and on $[-1/k,1/k],$ let $q_k(t)$ be the piece-wise affine function satisfying the two following properties:

$(i)$ $q_k$ is affine on $[-1/k,0]$ and on $[0,1/k],$  $q_k(0)= 2 k g'(1/k) - q_k(1/k) \ge 1,$ 

$(ii)$ $q_k$ is continuous on $\R.$ \\
The value of $q(0)$ is in fact chosen such that 
$$\int_{-1/k}^{1/k} q_k(t) dt= g'(1/k) - g'(-1/k).$$
This property ensures the existence of a unique $\mathcal{C}^2$ convex function $g_k(t)$ on $\R$ satisfying $g_k(t) \equiv g(t)$ for $|t| \ge 1/k$ and $g''_k(t)= q_k(t).$  One can check that $g_k$ is uniformly bounded and $g_k$ converges to $g.$ Hence $g_k(f)$ converges weakly to $g(f)$ as currents. On the other hand, direct computations give
$$g''_k(f) \ge \min\{1/3, 2k \log 2 - 1\}=1/3, \quad   |g'_k(t)| \le \log 3+ 2 \le 4$$
for $|t| \le 1$  and 
$$dd^c g_k(f)= g''_k(f) df \wedge d^c f + g'_k(f) dd^c f \ge  12^{-1} \big(df \wedge d^cf -  2n \|D^2 f\|_{L^{\infty}} \, \omega \big).$$
%for $c$ small enough independent of $k.$ Taking $k \rightarrow \infty,$ we get $dd^c g(f) \ge 0.$ Combined with the continuity of $g(f)$,  we deduce the desired result. 
The proof is finished.     
\end{proof}

\begin{proposition}  \label{pro_volumepshonRnddcnew} 
Let $\varphi$ be a p.s.h. function on an open set $V \subset \C^n.$ Let $A$ be a generic $\mathcal{C}^3$ submanifold of dimension $n$ of $V.$ Let $A_1$ be a compact subset of $A$ and for $\epsilon>0,$ let $A_{1, \epsilon}$ be the set of points in $\C^n$ of distance $\le \epsilon$ to $A_1.$ Then there is a constant $c$ independent of $\varphi,\epsilon$ for which we have
\begin{align} \label{ine_volumeepsilonuddcnewpro}
\int_{A_{1,\epsilon}} dd^c \varphi \wedge \omega^{n-1}  \le c \epsilon^{n-1} \int_{V} |\varphi| ,
\end{align}
where $\omega$ is the canonical K\"ahler form of $\C^n.$
\end{proposition}

\begin{proof}  Let $\delta$ be a small positive number which will be chosen later. Observe that the question is local. By using a partition of unity and Lemma \ref{pro_localcoordinates}, it is enough to prove the desired result for the case where  $A$ is the graph of a $\mathcal{C}^3$ map $h$ over $\B_n(0,3 \delta)$ such $h(0)= Dh(0)= D^2 h(0)=0$ and $\|h\|_{\mathcal{C}^3}$ is bounded independently of chosen local charts (hence, in particular, independent of $\delta$); and $A_1$ is the part of the graph over $\overline{\B}_n(0,\delta).$  We can assume that 
$$A_{1,\epsilon}=\{ \mathbf{x}+ i \mathbf{y}: \mathbf{x} \in \overline{\B}_n(0,\delta), |\mathbf{y}- h(\mathbf{x})| \le \epsilon\}$$
and $V=\B_n(0,3 \delta)+ i \B_n.$  

 Let  $g$ be the function defined in Lemma \ref{le_tinhcuthe}.  We write  $\mathbf{z}=(z_1, \cdots, z_n)$, $\mathbf{y}=(y_1, \cdots, y_n)$ and $h=(h_1, \cdots, h_n).$ Since $|D^2 h| \lesssim \delta$ on $\B_n(0,3\delta),$ one has
$$|D^2\big(y_j- h_j(\mathbf{x})\big) \| \lesssim \delta$$
for $1 \le j \le n.$  Using this and  Lemma \ref{le_tinhcuthe}, we see that the function $$\rho(\mathbf{z}):=\sum_{j=1}^n g\big(y_j - h_j(\mathbf{x})\big)$$
satisfies 
$$dd^c \rho \ge  \sum_{j=1}^n \big( \frac{i}{4\pi} d z_j \wedge d \bar{z}_j - \delta  M  d z_j \wedge d \bar{z}_j \big),$$
for some constant $M$ independent of $\delta.$ Thus  if $\delta$ is small enough independently of $\epsilon,$ the function $\rho$ is p.s.h. on $V.$   It is clear that $A_{1,\epsilon} \subset  A_1 \cap \{\rho \le 2n \epsilon\}.$ Let $\varphi_1(\mathbf{z}):= |\mathbf{y}- h(\mathbf{x})|^2.$ A direct computation shows that $\varphi_1$ is also p.s.h. on $V.$ Note that $|\varphi_1| \lesssim \epsilon^2$ on $\{ \rho \le 2 \epsilon.\}$  Now applying Lemma \ref{le_ddcmuntru1moipp} to $(\rho, \varphi_1)$ and to $T:= dd^c \varphi,$ we obtain     
$$\int_{ A_{1,\epsilon}} dd^c \varphi \wedge (dd^c \varphi_1)^{n-1} \lesssim \epsilon^{n-1} \| dd^c \varphi\|_{ \B_n(0,2\delta)+ i \B_n(0,1/2)} \lesssim \epsilon^{n-1} \int_V |\varphi|.$$
The last inequality together with the fact that $dd^c \varphi_1 \gtrsim \omega$ gives the desired result. The proof is finished.
\end{proof}

Note that a similar technique was used by Sibony in \cite{Sibony_duke} when dealing with the extension of positive closed currents (or more generally  pluripositive currents) through a CR submanifold.   For $\epsilon\in (0,1],$ let $K'_{\epsilon}$ be as above.  The following is just a direct consequence of Proposition \ref{pro_volumepshonRnddcnew}.
 
\begin{corollary}  \label{co_volumepshonRnddcnew} 
Let $V$ be an open subset of $\C^n$ containing $K'_1.$  Let $\varphi$ be a p.s.h. function on $V.$ Then  there is a constant $c$ independent of $\varphi, \epsilon$ for which we have
\begin{align} \label{ine_volumeepsilonuddcnew}
\int_{K'_{\epsilon}} dd^c \varphi \wedge \omega^{n-1}  \le c \epsilon^{n-1} \int_{V} |\varphi| ,
\end{align}
where $\omega$ is the canonical K\"ahler form of $\C^n.$
\end{corollary}

Now we are going to give some applications of these above estimates to our present problem. Firstly, we prove some auxiliary lemmas. Let $t_3, \tilde{\epsilon}_0$ and $\tilde{\theta}_0$ be the constants in Proposition \ref{pro_corverCndung Psi2}. Let $\tilde{F}$ be the family of analytic discs defined there. For simplicity, from now on, we denote $\tilde{F}(z, \boldsym{\tau},t_3)$ by $\tilde{F}(z, \boldsym{\tau}).$ Recall that the image of $\tilde{F}$ is contained in $\D^n.$ Put $\tilde{\epsilon}'_0:= t_3 \tilde{\epsilon}_0.$ 
 
 \begin{lemma} \label{le_tichphanRnfibration}  There exists  a positive constant $c_0$ such that for any Borel function $g$ on $\D^n,$ we have 
\begin{align} \label{ine_tichphangF}
\int_{\B_n(0, \epsilon'^*_0)} |g\big(\mathbf{x}, h(\mathbf{x})\big)|  \le c_0 \int_{[e^{-i\theta^*_0}, e^{i\theta^*_0}] \times \B_{n-1}^2} |g \circ \tilde{F}(e^{i\theta}, \boldsym{\tau})|.
\end{align}
\end{lemma}

\begin{proof} This is a direct consequence of Property $(i)$ of Proposition \ref{pro_corverCndung Psi2} and the change of variables theorem.  The proof is finished.
\end{proof}
 
 %PHAN NAY LAM DO QUA?  NHIEU CHO KHONG CHINH XAC? PHAI UOC LUONG GAN MANIFOLD K CHU KHONG PHAI GAN Rn
 
%By (\ref{ine_distancetoKhF*}) there is constant $d_1$ for which 
%\begin{align} \label{ine_FKht3d1}
%\dist\big(\tilde{F}(z, \boldsym{\tau}), K'\big) \lesssim d_1(1 - |z|),
%\end{align}
%for any $(z, \boldsym{\tau}).$ 
%For $\epsilon>0,$ let $K'_{\epsilon}$ be as in the last subsection.% the subset of $\D^n$ consisting of $\mathbf{z}$ such that $\dist(\mathbf{z}, K') \le d_1 \epsilon.$ 

\begin{lemma} \label{cor_analyticdisctau23tichphang} Let $g$ be a Borel function  on $\D^n.$

$(i)$ If $n=1,$ then
\begin{align*}
   \int_{\D \times \B_{n-1}(0, 1)^2} |g\circ \tilde{F}(z, \boldsym{\tau})|  \le  c_1 \int_{\D^n} |g(\mathbf{z})| ,
\end{align*}
for some constant $c_1$ independent of $g.$

$(ii)$ Assume $n>1.$  Let $t_0$ and $\delta_0$ be two positive real numbers such that $t_0 + \delta_0 > n-1 > \delta_0.$ Let 
$$M_g:= \sup_{\epsilon \in (0, 1)} \epsilon^{-t_0}\int_{K'_{\epsilon}} |g(\mathbf{z})|  $$
and $\lambda_0:= t_0 + \delta_0 -n +1.$ Assume $M_g < \infty.$ Then  we have 
\begin{align*}
   \int_{\D \times \B_{n-1}^2} (1- |z|)^{\delta_0} |g\circ \tilde{F}(z, \boldsym{\tau})|   \le \frac{2^{t_0}c_1 M_g}{\lambda_0}  \bigg[\int_{\D^n} |g(\mathbf{z})|  \bigg]^{\frac{\lambda_0}{t_0}},
\end{align*}
for some constant $c_1$ independent of $g,t_0,\delta_0.$
\end{lemma}

\begin{proof} When $n=1,$ the desired inequality is a direct application of the change of variables theorem and (\ref{ine_detDF}). Consider now $n>1.$  Put $Y:= \D \times  \B_{n-1}^2.$ Let $\epsilon$ be a small positive number which will be chosen later. Set 
$$Y_{\epsilon,0}:=\{(z, \boldsym{\tau})\in Y:  \dist\big(\tilde{F}(z, \boldsym{\tau}), K'\big) \ge \epsilon\}$$
and
$$Y_{\epsilon,k}:=\{(z, \boldsym{\tau})\in Y:  2^{-k}\epsilon \le   \dist\big(\tilde{F}(z, \boldsym{\tau}), K'\big)  \le 2^{-k+1}\epsilon\},$$
for $k \in \N.$ It is clear that $Y=  \cup_{k=0}^{\infty} Y_{\epsilon, k}.$ By definition of $K'_{\epsilon}$, we have
\begin{align} \label{inclusionF*Y} 
\tilde{F}(Y_{\epsilon, k}) \subset H_{2^{-k+1} \epsilon}.
\end{align}
 Denote by $vol_Y$ the canonical volume form on $Y.$ Write 
\begin{align} \label{eq_biendoigFtrenY}
 \int_{Y}  (1- |z|)^{\delta_0} |g\circ \tilde{F}| \, d vol_Y &= \sum_{k=0}^{\infty} \int_{Y_{\epsilon,k}}  (1- |z|)^{\delta_0}  |g\circ \tilde{F}| \, d vol_Y\\
 \nonumber
& \lesssim \sum_{k=0}^{\infty}  \int_{Y_{\epsilon,k}}  (1- |z|)^{\delta_0} |g\circ \tilde{F}(z, \boldsym{\tau})|  \, \frac{|\det D \tilde{F}(z, \boldsym{\tau})|}{\dist^{n-1}\big(\tilde{F}(z, \boldsym{\tau}), K'\big)} \, d vol_Y \\
\nonumber
& (\text{by (\ref{ine_detDF})})\\
\nonumber
&\lesssim    \sum_{k=0}^{\infty}  (2^{-k}\epsilon)^{-n+1+ \delta_0}   \int_{Y_{\epsilon,k}}  |g\circ \tilde{F}|  \, |\det D \tilde{F}| \, d vol_Y,
\end{align}
by definition of $Y_{\epsilon,k},$ (\ref{ine_distancetoKhF*}) and the fact that $-n+1 + \delta_0<0$. By change of variables, the last integral equals
$$ \int_{\tilde{F}(Y_{\epsilon,k})}  |g|$$
which is, for $k \ge 1,$ less than or equal to  
$$\int_{H_{2^{-k+1} \epsilon}} |g|  \le  (2^{-k+1}\epsilon)^{t_0} M_g$$
by definition of $M_g$ and (\ref{inclusionF*Y}).  This coupled with (\ref{eq_biendoigFtrenY}) yields that 
\begin{align} \label{eq_biendoigFtrenY2}
 \int_{Y}  (1- |z|)^{\delta_0} |g\circ \tilde{F}|  &\lesssim \epsilon^{-n+1+ \delta_0}  \int_{\D^n} |g|   + 2^{t_0} \epsilon^{\lambda_0} M_g  \sum_{k=1}^{\infty}  2^{-k \lambda_0}\\
 \nonumber
& \lesssim  \epsilon^{-n+1+ \delta_0}  \int_{\D^n} |g| +  \frac{2^{t_0} \epsilon^{\lambda_0} M_g}{2^{\lambda_0}-1}.
 \end{align}
 Choose $\epsilon= \|g\|^{1/t_0}_{L^1(\D^n)}.$ Using (\ref{eq_biendoigFtrenY2}) and the fact that $2^{\lambda_0} \ge 1 + \lambda_0,$ we get the desired inequality. The proof is finished.
\end{proof}
 
The following  will be crucial for our later purpose.

\begin{corollary} \label{cor_khatichcuaddc} Let $V$ be an open subset of $\C^n$ containing $K_1' \cup \overline{\D}^n.$ Let $\varphi$ be a p.s.h. function on $V.$ Let $\delta \in (0,1).$ Define $\gamma:= \delta/(n-1)$ if $n>1$ and $\gamma=1$ otherwise. Then  we have 
\begin{align}\label{cor_khatichcuaddccai1}
\int_{\D \times   \B_{n-1}^2} (1- |z|)^{\delta} dd^c (\varphi\circ \tilde{F})(z, \boldsym{\tau})  \lesssim_{\delta}  \|\varphi\|_{L^1(V)}^{\gamma}.
\end{align}
Furthermore, we have 
\begin{align} \label{cor_khatichcuaddccai2}
\int_{\{1 - 2 \epsilon \le |z| \le 1\} \times   \B_{n-1}^2} (1- |z|) dd^c (\varphi\circ \tilde{F})(z, \boldsym{\tau}) \lesssim_{\delta} \epsilon^{1- \frac{\delta(n-1)}{\delta+n-1}} \max\{ \|\varphi\|_{L^1(V)}^{\gamma}, \|\varphi\|_{L^1(V)} \},  
\end{align}
for every $\epsilon \in (0,1).$
\end{corollary}

\begin{proof}  Firstly we prove (\ref{cor_khatichcuaddccai1}). The case where $n=1$ is clear. Consider $n>1.$  Let $V_1 \Subset V$ be an open subset of $V.$ Fix a decreasing sequence of smooth p.s.h. functions $\varphi_l$ converging pointwise to $\varphi$ on $V_1$ and $\|\varphi_l\|_{L^1(V_1)} \le 2 \|\varphi\|_{L^1(V)}.$ Let $\delta \in (0,1).$  Since 
$$dd^c \varphi_l= \frac{i}{\pi} \sum_{1 \le  j, k \le n} \frac{\partial^2 \varphi_l}{\partial z_j \partial \bar{z}_k} d z_j \wedge d \bar{z}_k \ge 0,$$
using Corollary \ref{co_volumepshonRnddcnew}, there is a positive constant $c$ independent of $\varphi$ such that for every $j,k,l$ we have
\begin{align} \label{ine_volumeepsilonuddcvn}
\int_{ K'_{\epsilon}} \big| \frac{\partial^2 \varphi_l}{\partial z_j \partial \bar{z}_k}  \big|  \le c \epsilon^{n-1} \int_{V_1} |\varphi_l| \le c \epsilon^{n-1}\|\varphi\|_{L^1(V)}
\end{align}
which infers that the constant $M_g,$ defined in Lemma \ref{cor_analyticdisctau23tichphang} for 
$$g:=\frac{\partial^2 \varphi_l}{\partial z_j \partial \bar{z}_k}, \quad t_0= n-1, \quad \delta_0=\delta,$$
is finite.  Hence applying that lemma to the above mentioned $(g,t_0,\delta_0)$ gives
\begin{align} \label{ine_daywl_ddc}
\int_{\D \times  \B_{n-1}^2} (1- |z|)^{\delta} dd^c (\varphi_l\circ \tilde{F})(z, \boldsym{\tau})  \lesssim_{\delta} \|\varphi\|_{L^1(V)}^{\frac{\delta}{n-1}}.
\end{align}
On the other hand, since $dd^c \varphi_l \circ \tilde{F}$ converges weakly to $dd^c \varphi_l \circ \tilde{F}$ on $\D,$ we have
\begin{align} \label{ine_vlvcontinue}
 \liminf_{l \rightarrow \infty}\big\langle dd^c \big(\varphi_l \circ \tilde{F}(\cdot, \boldsym{\tau}) \big), f  \big \rangle \ge \big \langle dd^c \big(\varphi \circ \tilde{F}(\cdot, \boldsym{\tau})\big), f \big \rangle,
\end{align} 
for every positive continuous function $f$ on $\D.$  Letting $l \rightarrow \infty$ in (\ref{ine_daywl_ddc}) and then using (\ref{ine_vlvcontinue}) and Fatou's lemma, we get the desired result.

Now we prove (\ref{cor_khatichcuaddccai2}). As above, it is enough to prove it for $\varphi$ smooth. Set $W:=\{1 - 2 \epsilon \le |z| \le 1\} \times   \B_{n-1}^2.$ Let $r$ be a positive constant. Denote by  $W_1$ the subset of $W$ containing $(z, \boldsym{\tau})$ with $\dist \big(\tilde{F}(z,\boldsym{\tau}), K'\big) \ge r$ and by $W_2$ the complement of $W_1$ in $W.$ Let $\epsilon$ be a positive constant in $(0,1).$ Using (\ref{ine_detDF}) and the change of variables by $\tilde{F}$ on $W_1$ gives
$$\int_{W_1}(1- |z|)  dd^c (\varphi\circ \tilde{F})  \lesssim \epsilon \int_{W_1} dd^c (\varphi\circ \tilde{F}) \lesssim \epsilon r^{-n+1} \int_{\tilde{F}(W_1)}dd^c \varphi \wedge \omega^{n-1} \lesssim \epsilon r^{-n+1} \|\varphi\|_{L^1(V)}.$$ 
 By the proof of Lemma \ref{cor_analyticdisctau23tichphang} applied to $g=\frac{\partial^2 \varphi_l}{\partial z_j \partial \bar{z}_k},t_0=n-1$ and $\delta_0=\delta$, we have 
\begin{align*}
\int_{W_2}(1- |z|)dd^c (\varphi\circ \tilde{F}) &\le \epsilon^{1- \delta} \int_{W_2}(1- |z|)^{\delta} dd^c (\varphi\circ \tilde{F})\\
& \lesssim_{\delta} \epsilon^{1- \delta}  \big[\int_{\tilde{F}(W_2)}dd^c \varphi \wedge \omega^{n-1}\big]^{\frac{\delta}{n-1}} \le \epsilon^{1- \delta} r^{\delta}\|\varphi\|_{L^1(V)}^{\gamma}
\end{align*}
by (\ref{ine_volumeepsilonuddcnew}) and the fact that $\tilde{F}(W_2)$ is contained in $K'_{r} \times \B_{n-1}^2.$ Choose $r:= \epsilon^{\frac{\delta}{\delta+n-1}}.$ Taking the sume of the last two inequalities gives (\ref{cor_khatichcuaddccai2}). The proof is finished.
\end{proof}

\section{H\"older continuity for super-potentials} \label{sec_superpotential}

Recall that $\cali{C}$ defined at Introduction is a compact subset of  the set of $\omega$-p.s.h. functions on $X$ with respect to $L^1$-topology. Hence there is a positive number $r_0$ such that  
$$\|\varphi_0\|_{L^1(X)} \le r_0 \quad  \text{and} \quad \| \max \{\varphi_1,\varphi_2\}\|_{L^1(X)} \le r_0,$$ 
for every $\varphi_0, \varphi_1, \varphi_2 \in \cali{C}.$  Let $\cali{C}'$ be the set of  $\omega$-p.s.h. functions $\varphi$ on $X$ such that $ \|\varphi\|_{L^1(X)} \le 2 r_0.$ In this section, we will finish the proof of Theorem  \ref{the_MAgeneric2}. In order to do so, we will prove the following which is actually equivalent to Theorem  \ref{the_MAgeneric2} by Lemma \ref{le_reduction_nonnegative} below. Remember that we are still assuming that $\dim K=n.$ Let $\tilde{K}$ be the compact subset of $K$ as in Theorem \ref{the_MAgeneric2}.

\begin{proposition} \label{pro_equitheoremMAgeneric} Let $\alpha$ be a positive number strictly less than $1/(3n).$ Then  for any $\varphi_1,\varphi_2 \in \cali{C}'$ such that $\varphi_1 \ge \varphi_2,$ we have 
\begin{align} \label{ine_wnonnegative}
\int_{\tilde{K}} (\varphi_1 -\varphi_2) d vol_K \le c \int_ X (\varphi_1 -\varphi_2) d vol_X + c \bigg(\int_ X (\varphi_1 -\varphi_2) d vol_X\bigg)^{\alpha},
\end{align}
where $c$ is a constant independent of $\varphi_1, \varphi_2.$
\end{proposition}

\begin{lemma} \label{le_reduction_nonnegative} Proposition \ref{pro_equitheoremMAgeneric} implies Theorem \ref{the_MAgeneric2}.
\end{lemma}

\begin{proof} Take $\varphi_1, \varphi_2 \in \cali{C}.$ Observe that $\max\{\varphi_1,\varphi_2\}, \varphi_1,\varphi_2 \in \cali{C}'$ and $\max\{\varphi_1,\varphi_2\} \ge \varphi_j$ for $j=1,2.$ Hence, we can apply (\ref{ine_wnonnegative}) to $\max\{\varphi_1,\varphi_2\}, \varphi_j$ with $j=1,2.$ Using these two inequalities and the fact that
$$|\varphi_1 -\varphi_2|= 2 \max\{\varphi_1, \varphi_2\} - \varphi_1 -\varphi_2$$
gives 
$$\|\varphi_1 -\varphi_2\|_{L^1(\boldsym{1}_{\tilde{K}}vol_K)} \lesssim \max\{\|\varphi_1 -\varphi_2\|_{L^1(X)},\|\varphi_1 -\varphi_2\|^{\alpha}_{L^1(X)} \}$$
which means that $\boldsym{1}_{\tilde{K}} vol_K$ has H\"older continuous super-potential with H\"older exponent $\alpha.$ The proof is finished.
\end{proof}

The remaining of this section is devoted to prove Proposition \ref{pro_equitheoremMAgeneric}.  By \cite{Blocki_kolo_regular}, it is enough to prove (\ref{ine_wnonnegative}) for $\varphi_1,\varphi_2$ smooth. We will firstly show that for any nonnegative $\mathcal{C}^2$ function $v$ on $\overline{\D},$ the integral of $v$ over $\partial \D$ can be bounded by a quantity of the $L^1$-norm of $v$ over $\D$ and some H\"older norm of its Laplacian. This together with the exponent estimates in the last section are the key ingredients in the proof of Proposition \ref{pro_equitheoremMAgeneric}. We will reuse the notations from Section 2 for $M= \overline{\D}.$ %The latter part of this section combines all the developed techniques to give a proof of Proposition \ref{pro_equitheoremMAgeneric}.

%With a light modification,  the below treatment still works for a larger class of subharmonic functions than the class of continuous ones. However, the last class is sufficient for our purpose. Let $\rho$ be the function introduced in Section $2$ in case where $M=\overline{\D}.$ It is clear that $\rho (z)= 1 -|z|$ for $z \in \D.$ %Let $\cali{L}$ be the set consisting of subharmonic functions $v$ on $\D$ such that $v$ is continuous on $\overline{\D}$ and  ????

\begin{lemma} \label{le_uocluonghieuuvtrendia}  Let $v$ be a nonnegative $\mathcal{C}^2$ functions on $\overline{\D}.$ Let $\beta \in (1,2).$ Then we have
\begin{align} \label{ine_hieuuvddc}
\int_{\partial\D} v d\xi \lesssim_{\beta} \| dd^c v \|_{\tilde{\mathcal{C}}^{-\beta}(\overline{\D})}+ \int_{\D} v.
\end{align}
\end{lemma}

\begin{proof}  %Put $T:= dd^c v_1 -dd^c v_2.$ This is a $(1,1)$-current of order $0$ on $\D$ and 
%$$\|T\|_{\tilde{\mathcal{C}}^{-\beta}(\overline{\D})} \le \|dd^c v_1\|_{\tilde{\mathcal{C}}^{-\beta}(\overline{\D})}+\|dd^c v_1\|_{\tilde{\mathcal{C}}^{-\beta}(\overline{\D})} < \infty.$$
%By Lemma \ref{le_chuanCttrenbien}, $T$ can be extended to be a bounded linear functional on $\tilde{\mathcal{C}}^{\alpha}(\overline{\D})$ (see Section \ref{sec_interpolation}) whose norm is bounded a constant times $\|T\|_{\tilde{\mathcal{C}}^{-\alpha}(\overline{\D})}.$  We also denote by $T$ this extended functional. 

 By Riesz's representation formula, we have
\begin{align} \label{eq_riesz_v}
v(z)=  \int_{-\pi}^{\pi} P( e^{i\theta}, z) v(e^{i\theta}) d\theta+  \int_{ \{|\eta|<1\}} \log \frac{|z - \eta|}{| 1 - z \bar{\eta}|} dd^c v,
\end{align}
for  $z \in \D,$ where $P(\xi, z)$ is the Poisson kernel given by 
$$P(\xi, z)= (2\pi)^{-1} (|\xi|^2 - |z|^2)|\xi - z|^{-2}.$$  This implies that 
\begin{align} \label{eq_tichphanD1/2v1v2}
\int_{\D_{1/2}} v(z) &= \int_{-\pi}^{\pi} v(e^{i\theta}) d\theta \int_{ z \in\D_{1/2}} P(e^{i\theta}, z)+ \\
\nonumber
& +\int_{\D} dd^c v(\eta)  \int_{z \in \D_{1/2}}  \log \frac{|z - \eta|}{| 1 - z \bar{\eta}|}.
\end{align}
Set 
$$f(\eta):=\int_{ \{|z| <1/2\}}  \log \frac{|z - \eta|}{| 1 - z \bar{\eta}|}=\int_{ \{|z| <1/2\}}  \log  |z - \eta| - \int_{ \{|z| <1/2\}}  \log | 1 - z \bar{\eta}| .$$
Observe that $f(e^{i\theta})=0$ because 
$$\log \frac{| z- e^{i\theta}|}{| 1 - z e^{-i\theta}|}=0$$
for any $z \in \D.$ This means that $f|_{\partial \D}\equiv 0.$ We claim that $f$ is indeed in $\tilde{\mathcal{C}}^{\beta}(\overline{\D})$ for every $\beta \in (1,2).$ Since $z \in \D_{1/2}$ and $\eta\in \D,$ the function 
$$ \int_{\D_{1/2}}  \log | 1 - z \bar{\eta}| dx dy $$
 is smooth in $\eta \in \D.$ Hence, we only need to take care of  $\int_{z \in \D_{1/2}}  \log  |z - \eta|.$ It is clear that 
\begin{align} \label{eq_daohmaeta_f}
\partial_{\eta} \int_{ z \in \D_{1/2}}  \log  |z - \eta|= -\frac{1}{2} \int_{ z\in \D_{1/2}}  \frac{\bar{z}- \bar{\eta}}{|z - \eta|^2}=-\frac{1}{2} \int_{ z\in \D_{1/2}}  \frac{1}{z-\eta} \cdot
\end{align}
Let $g$ be the right-hand side of the last equation. We will show that $g \in \mathcal{C}^{\alpha}(\overline{\D})$ for every $\alpha \in (0,1).$ If we can do so, then $\partial_{\eta} f \in \mathcal{C}^{\alpha}(\overline{\D}),$ using similar argument we also gets $\partial_{\bar{\eta}} f \in \mathcal{C}^{\alpha}(\overline{\D}),$ hence $f \in \tilde{\mathcal{C}}^{\beta}(\overline{\D})$ for $\beta \in (1,2).$ Let  $\alpha \in (0,1).$ For $(\eta,\eta') \in \D^2,$ consider the difference
\begin{align} \label{ine_hieulog}
\big|  \frac{1}{z-\eta}-  \frac{1}{z-\eta'} \big| &= \big| \frac{\eta - \eta'}{(z- \eta)(z- \eta')} \big| \le   \frac{|\eta - \eta'|^{\alpha}}{|(z- \eta)(z- \eta')|^{\alpha}} \bigg|  \frac{1}{z-\eta}-  \frac{1}{z-\eta'} \bigg|^{1 - \alpha}\\
\nonumber
& \le |\eta - \eta'|^{\alpha}\big[\frac{1}{|z- \eta|\, |z- \eta')|^{\alpha}}+ \frac{1}{ |z- \eta)|^{\alpha}|z- \eta|}\big]. 
\end{align}
%Observe that for any $x \in (0, \infty),$ we have $\log(1+ x) \lesssim_{\beta} |x|^{\beta}.$  Applying this to the right-hand side of (\ref{ine_hieulog}) gives 
%$$\big| \log |z- \eta |- \log |z- \eta'| \big|  \lesssim_{\beta} \max \big\{ \big|\frac{\eta-\eta'}{z- \eta}\big|^{\beta}, \big|\frac{\eta-\eta'}{z- \eta'}\big|^{\beta}  \big\}.$$
It is not difficult to see that the integration of the right-hand side of (\ref{ine_hieulog}) over $z \in \D_{1/2}$ is bounded by $|\eta - \eta'|^{\alpha}$ times a constant depending only on $\alpha$. Thus one gets $g \in \mathcal{C}^{\alpha}(\overline{\D}).$ As explained above, this yields $f \in \tilde{\mathcal{C}}^{\beta}(\overline{\D}).$ The last property combined with (\ref{eq_tichphanD1/2v1v2}) gives 
\begin{align} \label{eq_tichphanD1/2v1v23}
\bigg|\int_{-\pi}^{\pi} v(e^{i\theta}) d\theta \int_{ z\in \D_{1/2}} P(e^{i\theta}, z) \bigg| \le  \|v\|_{L^{1}(\D_{1/2})} +\|dd^c v\|_{\tilde{\mathcal{C}}^{-\beta}(\overline{\D})} \|f\|_{\tilde{\mathcal{C}}^{\beta}(\overline{\D})}.
\end{align}
 By our hypothesis that $v \ge 0$ and the fact that $P(e^{i\theta}, z) \gtrsim 1$ for $z \in \D_{1/2},$ using (\ref{eq_tichphanD1/2v1v23}), one obtains that 
 \begin{align} \label{eq_tichphanD1/2v1v234}
\int_{\partial  \D} v d\xi  \lesssim_{\beta}  \|v\|_{L^{1}(\D_{1/2})} +\|dd^c v \|_{\tilde{\mathcal{C}}^{-\beta}(\overline{\D})}.
\end{align}
 The proof is finished.
\end{proof}

\begin{proposition} \label{pro_uocluonghieuuvtrendia}  Let $v$ be a nonnegative $\mathcal{C}^2$ function on  $\overline{\D}.$ Let $\epsilon,\beta_0\in (0,1)$ and $\beta \in (1,2).$ Let $\gamma$ be the unique real number for which $\beta= \gamma \beta_0 + (1- \gamma) 2.$  Then  we have
\begin{multline} \label{ine_hieuuvmualpha}
\int_{\partial\D} v  d\xi \lesssim_{(\beta_0, \beta)} \| v \|_{L^1(\D)}+  \epsilon^{-2(1- \gamma)} \|dd^c v  \|^{\gamma}_{\tilde{\mathcal{C}}^{-\beta_0}(\overline{\D})} \| v \|^{1- \gamma}_{L^1(\D)}+\\
+ \|dd^c v  \|^{\gamma}_{\tilde{\mathcal{C}}^{-\beta_0}(\overline{\D})}\big[\int_{1- 2 \epsilon \le |z|\le 1}(1- |z|) |dd^c v| \big]^{1-\gamma}.
\end{multline}
\end{proposition}

\begin{proof} Firstly we will estimate  $\| dd^c v \|_{\tilde{\mathcal{C}}^{-2}(\overline{\D})}.$ Let $\chi \in \mathcal{C}^{\infty}(\R)$ such that $0 \le \chi \le 1,$ $\chi \equiv 0$ on $[-1,1]$ and $\chi \equiv 1$ outside $[-2,2].$ For $\epsilon \in (0,1),$ put $\chi_{\epsilon}(z):= \chi(\frac{1- |z|}{ \epsilon})$ for $z \in \D.$ We have $\supp\chi_{\epsilon} \subset \{z: |z| \le 1 - \epsilon\}$ and $\chi_{\epsilon}(z)= 1$ for $z$ with $|z| \le 1- 2\epsilon.$ Let $\Phi$ be a function in $\tilde{\mathcal{C}}^{2}(\overline{\D})$ with $\|\Phi\|_{\mathcal{C}^2} \le 1.$  Since $\Phi \equiv 0$ on $\partial \D$ we have $|\Phi(z)| \le 1- |z|.$ Decompose
$$\langle  dd^c v, \Phi \rangle= \langle  dd^c v, \chi_{\epsilon}\Phi \rangle+\langle  dd^c v, (1- \chi_{\epsilon})\Phi \rangle.$$
Denote by $I_1,I_2$ respectively the first and second terms in the right-hand side of the last equality. By properties of $\Phi$ and $\chi_{\epsilon},$ one gets
$$|I_2| \le  2\int_{1- 2 \epsilon \le |z|\le 1}(1- |z|) |dd^c v|.$$
On the other hand, performing an integration by parts gives
$$|I_1| \le \int_{\D} |v dd^c(\chi_{\epsilon} \Phi) | \lesssim \epsilon^{-2}\int_{\D} |v|.$$
Hence, we obtain
\begin{align} \label{ine_chuanChaicuav}
\| dd^c v \|_{\tilde{\mathcal{C}}^{-2}(\overline{\D})}= \sup_{\{\Phi \in \tilde{\mathcal{C}}^2(\D): \|\Phi\|_{\mathcal{C}^2} \le 1 \}}\big| \langle  dd^c v, \Phi \rangle\big| \lesssim \epsilon^{-2}\int_{\D} |v|+ \int_{1- 2 \epsilon \le |z|\le 1}(1- |z|) |dd^c v|.
\end{align}
Now applying Proposition \ref{pro_sosanhdistancevoibien} to $dd^c v$ and $M= \overline{\D},$ one gets 
\begin{align*}
\| dd^c v \|_{\tilde{\mathcal{C}}^{-\beta}(\overline{\D})} &\lesssim  \|dd^c v\|^{\gamma}_{\tilde{\mathcal{C}}^{-\beta_0}(\overline{\D})} \| dd^c v \|^{1- \gamma}_{\tilde{\mathcal{C}}^{2}(\overline{\D})}.
\end{align*}
The last inequality combined with (\ref{ine_hieuuvddc}) and (\ref{ine_chuanChaicuav}) gives (\ref{ine_hieuuvmualpha}).  The proof is finished.
\end{proof}

%Note that a similar version of the last proposition also holds if $v$ is a continuous p.s.h function on $\overline{\D}.$ In this case, most of arguments can be applied without changes, the only trouble is the fact that $dd^c v$ may have an infinite mass on $\D.$ Since this is not necessary for our purposes, we will not treat that situation here.  
%\begin{remark} \label{re_smooth} Riesz's representation formula also holds for smooth functions on $\overline{\D}.$ Hence, if we have $v_1$ smooth, then we can remove the condition on subharmonicty of $v_1.$ The same remark is applied to $v_2.$
%\end{remark}

%Let $f \in \tilde{\mathcal{C}}^{\beta}(\overline{\D})$ of $\mathcal{C}^{\beta}$-norm equal $1.$ Let $z= |z| e^{i\theta} \in \D.$ We have $|f(|z| e^{i\theta}) - f(e^{i\theta})| \le (1- |z|).$ Since $f(e^{i\theta})=0,$ one  gets $|f(z)| \le (1 -|z|)^{\beta}.$ 

%\subsection{Local version of  Proposition \ref{pro_equitheoremMAgeneric}} \label{subsec_localversion14} 
We are now about to prove the local version of  Proposition \ref{pro_equitheoremMAgeneric}. Given a point $a \in K,$ a small open neighborhood $K'$ of $a$ in $K$ can be described as in Section \ref{sec_analyticdisc}. Namely, there are a $\mathcal{C}^3$ map $h$ from $\overline{\B}_n$ to $\R^n$ with $h(0)= Dh(0)=0$ and local holomorphic coordinates in $X$ such that 
$$K':= \{\mathbf{x}+i h(\mathbf{x} ) :  \mathbf{x} \in  \B_n\} \subset \D_2^n.$$
%Let $\tilde{F}(z, \boldsym{\tau})$ be the family of analytic discs constructed in Proposition \ref{pro_corverCndung Psi2}. 
Let  $\tilde{F}(z, \boldsym{\tau})$, $t_3, \tilde{\epsilon}_0'$ and $\tilde{\theta}_0$ be as in  Section \ref{sec_intergrability}. The couple $(K',\D_2^n)$ is considered as the local counterpart of $(K,X).$ One can replace $\D_2^n$ by any polydisks $\D_r^n$ with $r>1$ without making any differences in what follows.  

 Let $\beta_0 \in (0,1).$ For every positive continuous $(1,1)$-current $T$ on an open neighborhood of $\overline{\D}$, we have
\begin{align} \label{ine_chuanTCbeta}
 \|T\|_{\tilde{\mathcal{C}}^{-\beta_0}(\overline{\D})} \le \int_{\D} (1- |z|)^{\beta_0} T.
\end{align} 
%Let $\gamma:= \beta_0/(n-1)$ if $n>1$ and $\gamma:=1$ otherwise. 
 Let $\varphi_1$ and $\varphi_2$ be two $\mathcal{C}^2$ p.s.h. functions on $\D_2^n$ such that $\varphi_1 \ge \varphi_2$ and $\|\varphi_j\|_{L^1(\D_2^n)} \le 1$ for $j=1,2.$  Put $\varphi:= \varphi_1 -\varphi_2$ which is $\mathcal{C}^2$ and nonnegative.  Define 
\begin{align*}
g_1(\boldsym{\tau}):= \| dd^c \big(\varphi \circ \tilde{F}(\cdot, \boldsym{\tau})\big)\|_{\tilde{\mathcal{C}}^{-\beta_0}(\overline{\D})} 
\end{align*}
which is less than or equal to 
\begin{align} \label{eq_sumvarphi12}
\| dd^c \big(\varphi_1 \circ \tilde{F}(\cdot, \boldsym{\tau})\big)\|_{\tilde{\mathcal{C}}^{-\beta_0}(\overline{\D})}+ \| dd^c \big(\varphi_2 \circ \tilde{F}(\cdot, \boldsym{\tau})\big)\|_{\tilde{\mathcal{C}}^{-\beta_0}(\overline{\D})}.
\end{align}
Since $\tilde{F}$ is $\mathcal{C}^2,$ so is $\varphi_j \circ \tilde{F}$ for $j=1,2.$ Using (\ref{ine_chuanTCbeta}) for $T= dd^c \big(\varphi_j \circ \tilde{F}(\cdot, \boldsym{\tau})\big)$ and (\ref{cor_khatichcuaddccai1}),  we deduce that the integral of the sum (\ref{eq_sumvarphi12}) with respect to $\boldsym{\tau} \in \B_{n-1}^2$ is $\lesssim_{\beta_0} 1.$ This implies
\begin{align} \label{ine_tichphantaucuagw12}
\int_{\B_{n-1}^2} g_1(\boldsym{\tau}) d \boldsym{\tau} \lesssim_{\beta_0} 1.
\end{align}
 Put 
$$g_2(\boldsym{\tau}):= \| \varphi \circ \tilde{F}(\cdot, \boldsym{\tau})\|_{L^1(\D)}.$$
By Corollary \ref{cor_volumepshonRn}, the function $\varphi$  satisfy the hypothesis of Lemma \ref{cor_analyticdisctau23tichphang} for $\delta_0=0$ and $t_0= n-1+  \epsilon$ with $\epsilon \in (0,1).$ As a result, we get 
\begin{align} \label{ine_tichphantaucuagw123}
\int_{\B_{n-1}^2} g_2(\boldsym{\tau}) d \boldsym{\tau} \lesssim_{\epsilon}  \bigg[\int_{\D^n_{2}} |\varphi| \bigg]^{\frac{\epsilon}{n-1+ \epsilon}}.
\end{align}
For $\epsilon' \in (0,1),$ we  define 
$$g_3(\boldsym{\tau},\epsilon'):=  \int_{1- 2 \epsilon' \le  |z| \le 1}(1-|z|)dd^c \big(\varphi_1 \circ \tilde{F}(\cdot, \boldsym{\tau}) \big)+ \int_{1- 2 \epsilon' \le  |z| \le 1}(1-|z|)  dd^c\big(\varphi_2 \circ \tilde{F}(\cdot, \boldsym{\tau}) \big).$$
By (\ref{cor_khatichcuaddccai2}), we have
\begin{align} \label{ine_tichphantaucuagw123gba}
\int_{\B_{n-1}^2} g_3(\boldsym{\tau}, \epsilon') d \boldsym{\tau} \lesssim_{\delta} (\epsilon')^{1-\frac{\delta(n-1)}{n-1+ \delta}},
\end{align}
for any $\delta \in (0,1).$

\begin{proposition} \label{pro_tichphantrentauw1w2} Let $\varphi_1$ and $\varphi_2$ be two $\mathcal{C}^2$  p.s.h. functions on $\D_2^n$ such that $\varphi_1 \ge \varphi_2$ and $\|\varphi_j\|_{L^1(\D_2^n)} \le 1$ for $j=1,2.$ Let $\varphi:= \varphi_1 -\varphi_2.$  Then  we have 
\begin{align} \label{ine_superpotentiallocal}
\int_{\B_n(0,\epsilon'^*_0)} \varphi \big(\mathbf{x}, h(\mathbf{x})\big) d \mathbf{x} \lesssim_{\delta}  \| \varphi \|_{L^1(\D^n_2)}^{\frac{1}{3n} - \delta},
\end{align}
for any $\delta \in (0, \frac{1}{3n}).$
%where $$M(\varphi_1,\varphi_2):=\bigg[\int_{\D^n_{2}} \big(|\varphi_1|+ |\varphi_2|\big)  \bigg]^{\frac{\gamma(2- \beta)}{2-\beta_0}}.$$
\end{proposition}

\begin{proof}  Let $\epsilon, \epsilon', \beta_0 \in (0,1)$ and $ \beta \in (1, 2).$ Let $g_1,g_2,g_3$ be as above. Applying Lemma \ref{le_tichphanRnfibration} to $g= \varphi$ gives
$$\int_{\B_n(0,\epsilon'^*_0)} | \varphi \big(\mathbf{x}, h(\mathbf{x})\big)| d \mathbf{x} \lesssim \int_{\B_{n-1}^2 } d \boldsym{\tau} \int_{\partial \D} | \varphi \circ \tilde{F}(\cdot, \boldsym{\tau}) | d \xi.$$
Put $\gamma:=\frac{2 - \beta}{2 - \beta_0}.$ Applying  Proposition \ref{pro_uocluonghieuuvtrendia} to  $v= \varphi \circ \tilde{F}(\cdot, \boldsym{\tau}) \in \mathcal{C}^2$ shows that the right-hand side of the last inequality is 
$$ \lesssim_{(\beta_0,\beta)}  \int_{\B_{n-1}^2} g_2 d \boldsym{\tau}+  (\epsilon')^{-2(1- \gamma)} \int_{\B_{n-1}^2} g_1^{\gamma} g_2^{1- \gamma} d \boldsym{\tau}+\int_{\B_{n-1}^2} g_1^{\gamma} g_3^{1- \gamma}(\cdot, \epsilon') d \boldsym{\tau}.$$
The first term of the last sum is 
$$\lesssim_{\epsilon} \| \varphi \|^{\frac{\epsilon}{n-1+ \epsilon}}_{L^1(\D_{2}^n)}$$
by (\ref{ine_tichphantaucuagw123}). On the other hand,  by  the H\"older inequality, the second one is  $\le$ 
$$(\epsilon')^{-2(1- \gamma)}  \|g_1\|_{L^1}^{\gamma} \|g_2\|_{L^1}^{1- \gamma}$$
and the third one is $\le$ 
$$\|g_1\|_{L^1}^{\gamma} \|g_3(\cdot, \epsilon')\|_{L^1}^{1- \gamma},$$
where  the $L^1$-norm is taken over $\B_{n-1}^2.$ Taking into account  (\ref{ine_tichphantaucuagw12}) and (\ref{ine_tichphantaucuagw123}), one obtains
$$(\epsilon')^{-2(1- \gamma)} \|g_1\|_{L^1}^{\gamma} \|g_2\|_{L^1}^{1- \gamma} \lesssim_{\beta_0,\epsilon}   (\epsilon')^{-2(1- \gamma)} \| \varphi \|_{L^1(\D^n_2)}^{\frac{\epsilon(1- \gamma)}{n-1+ \epsilon}}.$$ 
By  (\ref{ine_tichphantaucuagw12}) and (\ref{ine_tichphantaucuagw123gba}),  we have
$$\|g_1\|_{L^1}^{\gamma} \|g_3(\cdot, \epsilon')\|_{L^1}^{1- \gamma}  \lesssim_{\beta_0,\delta} (\epsilon')^{(1-\frac{\delta(n-1)}{n-1+ \delta})(1- \gamma) }, $$
for every $\epsilon' \in (0,1).$ Put 
$$a_1:= \frac{\epsilon}{(n-1+ \epsilon)(3-\frac{\delta(n-1)}{n-1+ \delta}) }, \quad  a_2:= \frac{\epsilon(1-\frac{\delta(n-1)}{n-1+ \delta})(1-\gamma)}{(n-1+ \epsilon)(3-\frac{\delta(n-1)}{n-1+ \delta}) }\cdot$$
Choose  $\epsilon':= \| \varphi \|_{L^1(\D^n_2)}^{a_1}.$ Combining all these above inequalities, we get
$$\int_{\B_n(0,\epsilon'^*_0)} | \varphi \big(\mathbf{x}, h(\mathbf{x})\big)| d \mathbf{x} \lesssim_{(\beta_0,\beta,\delta,\epsilon)}  \| \varphi \|_{L^1(\D^n_2)}^{a_2}.$$
Observe that $a_2 \rightarrow \frac{1}{3n}$ as $\epsilon \rightarrow 1,$  $\beta \rightarrow 2,$ $\beta_0 \rightarrow 0$, $\delta \rightarrow 0.$ Thus, the proof is finished.
\end{proof}

\begin{proof}[End of the proof of Proposition \ref{pro_equitheoremMAgeneric} in the case where $\dim K=n$] Given any $a \in K,$ let $\tilde{F}_a$ and $\tilde{\epsilon}_a$ be as in   Proposition \ref{pro_corverCndung Psi2global}. Since $\tilde{K}$ is compact, we can cover it by a finite number of ball $B_K(a, \tilde{\epsilon}_a).$  Hence, in order to prove (\ref{ine_wnonnegative}), it is enough to restrict ourselves to local charts. In other words, we are now being in the situation with the model $(K', \D_2^n)$ described above. Moreover, by subtracting a suitable common smooth function, we can assume that $\varphi_1,\varphi_2$ in (\ref{ine_wnonnegative}) are $\mathcal{C}^2$ p.s.h. functions on $\D_2^n.$ Hence, the desired result follows directly from Proposition \ref{pro_tichphantrentauw1w2}.  The proof is finished.
\end{proof}

We now deal with the case where the dimension of $K$ is greater than $n.$   Let $n_K:= \dim K>n.$ Since $K$ is generic, we have $T_a K + J T_a K= T_a X,$ where $a \in K$ and $J$ is the complex structure of $X.$ We then deduce that $T_a K \cap J T_a K$ is of even dimension which equals $2 n_K- 2n.$ The codimension $d$ of $K$ equals $2n - n_K.$

\begin{proposition} \label{le_coordinatesKnK} 
Let $a$ be a point in $K.$  There exist local $\mathcal{C}^2$ coordinates $(W,\Psi)$ of $X$ around $a$ such that the following properties hold:

$(i)$  $\Psi: W \rightarrow  \C^{d}  \times \C^{n_K - n}$ is a $\mathcal{C}^2$ diffeomorphism onto its image which equals 
$$ \big( \B_{d} + i  \B_{d}(0,2) \big) \times \D^{n_K-n}$$
and $\Psi(p)=0$ and    $\Psi^{-1}(\mathbf{z}_1, \mathbf{z}_2)$ is holomorphic in $\mathbf{z}_1$ for every fixed $\mathbf{z}_2 \in \D^{n_K -n},$ 

$(ii)$  there is  a $\mathcal{C}^2$ map $h(\Re \mathbf{z}_1, \mathbf{z}_2)$  from $\overline{\B}_{d} \times \D^{n_K -n}$ to $\R^{d}$ so that  for every $\mathbf{z}_2$ fixed, $h(\cdot, \mathbf{z}_2) \in \mathcal{C}^3$ and 
\begin{align*} %\label{ine_nKdaohamhmoi}
D^j_{\Re \mathbf{z}_1} h(0, \mathbf{z}_2)=0
\end{align*}
 for $j=0$ or $1$  and  
$$ \Psi(K\cap W)=\big\{(\mathbf{z}_1, \mathbf{z}_2) \in \big( \B_{d} + i \R^{d} \big) \times \D^{n_K-n}: \Im \mathbf{z}_1= h(\Re \mathbf{z}_1, \mathbf{z}_2) \big\}.$$

%$(iii)$ the norm $\|h\|_{\mathcal{C}^2}$ is uniformly bounded in $a \in \tilde{K}.$

\end{proposition} 

\begin{proof} It is well-known that in suitable holomorphic local coordinates, $K$ is  given by 
$$K=\big \{(\mathbf{z}_1, \mathbf{z}_2) \in \big( \B_{d} + i \R^{d} \big) \times \D^{n_K-n}: \Im \mathbf{z}_1= \tilde{h}(\Re \mathbf{z}_1, \Re\mathbf{z}_2, \Im \mathbf{z}_2)  \big\}$$ 
 where $\tilde{h}$ is a $\mathcal{C}^3$ map of uniformly bounded $\mathcal{C}^3$ norm in $p$ and $\tilde{h}(0)= D\tilde{h}(0)=0,$ see \cite{Baouendi_Ebenfelt_Rothschild}. For $\mathbf{z}_2$ fixed, we choose the tangent space of the graph of $\tilde{h}(\cdot, \mathbf{z}_2)$ at $0$ and its orthogonal subspace as new holomorphic coordinates of $\C^{d}.$ These new coordinates depend $\mathcal{C}^2$ on (but in general  not holomorphically) on the parameter $\mathbf{z}_2.$ In  these new coordinates, one easily see that $K$ is given by the formula given in the asssertion $(ii)$ for some $\mathcal{C}^2$ map $h$ with the desired properties.
% $h(0, \mathbf{z}_2)= D_{(\Re \mathbf{z}_1)} h(0, \mathbf{z}_2)=0.$ To obtain (\ref{ine_nKdaohamhmoi}), one just need to apply Proposition \ref{pro_localcoordinates} for each fixed $\mathbf{z}_2$ and note that the biolomorphism $\Phi$ there can be chosen to be smooth with respect to $\mathbf{z}_2$. After each change of coordinates, the holomorphicity with respect to  $\mathbf{z}_1$ is preserved whereas the one for $\mathbf{z}_2$ is not guaranteed. 
The proof is finished.
\end{proof}

\begin{remark} \label{re_dimKlonhonn} As in Lemma \ref{pro_localcoordinates} we can obtain furthermore that $D^2_{\Re \mathbf{z}_1} h(0, \mathbf{z}_2)=0$ and $\|h(\cdot, \mathbf{z}_2)\|_{\mathcal{C}^3}$ is bounded uniformly in $a=(\mathbf{z}_1,\mathbf{z}_2) \in \tilde{K}$  but in this case we will lose a unit for the regularity in $\mathbf{z}_2,$ \emph{i.e} $\Psi$ and $h$ are only $\mathcal{C}^1$ in $\mathbf{z}_2.$
\end{remark}

Thanks to Proposition \ref{le_coordinatesKnK}, we can consider $K$ locally as a family of generic submanifolds of $\C^{d}$ of dimension $d$ parameterized by $\mathbf{z}_2 \in \D^{n_K -n}.$ This allows us to reduce the question to the previous case where we already dealt with generic submanifolds of minimal dimension. By compactness of $\tilde{K},$ we can cover it by local charts $W$ as in Proposition \ref{le_coordinatesKnK}. From now on, we work exclusively on a such local chart. Hence, we can identify $K$ with $\Psi(K \cap W).$ Let $h$ and $\Psi$ be as in that proposition. The map $h$ will be seen as  a family of  maps of $\mathbf{z}_1$ parameterized by $\mathbf{z}_2.$ For $\mathbf{z}_2 \in \D^{n_K -n},$ define
$$K'_{\mathbf{z}_2}:=\{\mathbf{z}_1 \in \big( \B_{d} + i \R^{d} \big):   \Im \mathbf{z}_1= h(\Re \mathbf{z}_1, \mathbf{z}_2) \big\}$$
which is identified with $K'_{\mathbf{z}_2} \times \{  \mathbf{z}_2\} \subset  \C^n.$ Then $K$ is  foliated by $K'_{\mathbf{z}_2}.$

We are now going to construct a family of analytic discs partly attached to $K.$ The strategy will be almost identical with what we did.  Let  $u_0$ be a function described in Lemma \ref{le_version8existenceu_MA} and $\theta_{u_0}$ be the constant there.  Let $\boldsym{\tau}_1, \boldsym{\tau}_2 \in \overline{\B}_{d-1} \subset \R^{d-1}.$ Define $\boldsym{\tau}^*_1:= (1, \boldsym{\tau}_1) \in \R^d$ and $\boldsym{\tau}^*_2:= (0, \boldsym{\tau}_1) \in \R^d$ and $\boldsym{\tau}:=(\boldsym{\tau}_1,\boldsym{\tau}_2).$ Let $t$ be a positive number in $(0,1).$ Consider the following modified version of the equation \ref{Bishoptype}: 
\begin{align}\label{Bishoptype2}
U_{\boldsym{\tau},\mathbf{z}_2,t}(\xi)= t\boldsym{\tau}^*_2 - \mathcal{T}_1\big(h(U_{\boldsym{\tau},\mathbf{z}_2,t};\mathbf{z}_2) \big)(\xi) - t\mathcal{T}_1 u_0(\xi) \, \boldsym{\tau}^*_1,
\end{align} 
where $U: \partial \D \rightarrow \B_{d}$ is H\"older continuous.  

%Write $\mathbf{z}_1= \mathbf{x}_1+ i \mathbf{y}_1.$ 
Since $h(0, \mathbf{z}_2)=D_{\Re\mathbf{z}_1} h(0, \mathbf{z}_2)=0$ for every $\mathbf{z}_2,$ we can use the same reason mentioned in the proof of Proposition \ref{pro_BishopequationMA} to show that if $t$ is small enough, the equation (\ref{Bishoptype2}) has a unique  solution $U_{\boldsym{\tau},\mathbf{z}_2,t}$ in $\mathcal{C}^{2,1/2}(\partial \D \times \B_{d-1}^2)$ for $\mathbf{z}_2$ fixed so that $U_{\boldsym{\tau},\mathbf{z}_2,t} \in \mathcal{C}^1$ as a function of $(z, \boldsym{\tau}, \mathbf{z}_2)$. We use the same notation $U_{\boldsym{\tau},\mathbf{z}_2,t}$ to denote the harmonic extension of $U_{\boldsym{\tau},\mathbf{z}_2,t}$ to $\D.$  Let $P_{\boldsym{\tau},\mathbf{z}_2,t}(z)$ be  the harmonic extension of $h\big(U_{\boldsym{\tau},\mathbf{z}_2,t}(\xi), \mathbf{z}_2\big)$ to $\D.$ %One should not confuse $P_{\mathbf{z},t}(z)$ with $h\big(U_{\mathbf{z},t}(z)\big).$  
Define
$$F(z, \boldsym{\tau},\mathbf{z}_2,t) := U_{\boldsym{\tau},\mathbf{z}_2,t}(z)+ i  P_{\boldsym{\tau},\mathbf{z}_2,t}(z)+ i  t \, u_0(z) \, \boldsym{\tau}^*_1$$
which is a family of analytic discs to $\C^d$ parametrized by $(\boldsym{\tau},\mathbf{z}_2,t).$ By our choice of $u_0,$  we have $F(\xi, \boldsym{\tau},\mathbf{z}_2,t) \in K_{\mathbf{z}_2}$ for  $\xi \in [e^{-i\theta_{u_0}}, e^{i\theta_{u_0}}].$ Now define
$$F'(z, \boldsym{\tau},\mathbf{z}_2,t) := \big(F_{\boldsym{\tau},\mathbf{z}_2,t}(z), \mathbf{z}_2\big) \in \C^n$$
which is a family of analytic discs to $X$ partly attached to $K.$ Here we used an essential fact that the $\mathcal{C}^2$ coordinates $(\mathbf{z}_1, \mathbf{z}_2)$ are holomorphic in $\mathbf{z}_1.$ Proposition \ref{pro_corverCndung Psi} with $n$ replaced by $d$ implies that for two positive constants $(t,r_0)$ small enough, $F'$ is a diffeomorphism on 
$$\big(\B_2(1,r_0) \cap \D\big) \times \overline{\B}_{d-1}^2 \times  \D^{n_K-n}$$
and its differential satisfies 
\begin{align} \label{ine_detDF0nK}
\big|\det D F'(z, \boldsym{\tau},\mathbf{z}_2,t) \big| \gtrsim t^{d+1} \dist^{d-1}\big(F'(z, \boldsym{\tau},\mathbf{z}_2,t), K'_{\mathbf{z}_2}  \big) \gtrsim t^{2d} (1- |z|)^{d-1}.
\end{align}
Now applying the same arguments right before Proposition \ref{pro_corverCndung Psi2}, one gets the following.

\begin{proposition} \label{pro_corverCndung Psi2nK} There  exists a map $\tilde{F}: \D \times \B_{d-1}^2 \times \D^{n_K- n} \rightarrow X$ which is a diffeomorphism  onto its image such that the following three properties hold:

$(i)$ there are positive constants $\tilde{\theta}_0$ and $\tilde{\epsilon}_0$ so that for every $\boldsym{\tau}_1 \in \overline{\B}_{d-1}$ the restriction map $\tilde{F}(\cdot, \boldsym{\tau}_1): [e^{-i\tilde{\theta}_0}, e^{i \tilde{\theta}_0}] \times \overline{\B}_{d-1} \times \D^{n_K -n} \rightarrow K$ is a diffeomorphism onto its image which contains the graph of $h$ over $\B_d(0, \tilde{\epsilon}_0) \times \D^{n_K -n},$     

$(ii)$ $\tilde{F}(\cdot, \boldsym{\tau},\mathbf{z}_2)$ is an analytic disc to $X$ and
\begin{align} \label{ine_detDFnK}
\big|\det D \tilde{F}(z, \boldsym{\tau},\mathbf{z}_2) \big| \gtrsim \dist^{d-1}\big(\tilde{F}(z, \boldsym{\tau},\mathbf{z}_2,t), K'_{\mathbf{z}_2}  \big)  \gtrsim  (1- |z|)^{d-1}.
\end{align}
\end{proposition}

Proposition \ref{pro_corverCndung Psi2nK} and Remark \ref{re_dimKlonhonn}  allow us to repeat all of arguments in the proof of Theorem \ref{the_MAgeneric2} in the case where $n_K=n$ for our present situation. Hence, this finishes the proof of Theorem \ref{the_MAgeneric2}.

\bibliography{mongeampere}
\bibliographystyle{siam}

\end{document}